\documentclass[reqno,11pt,letterpaper]{amsart}
\usepackage{amsmath,amssymb,amsthm,graphicx,mathrsfs,url}
\usepackage[usenames,dvipsnames]{color}
\usepackage[colorlinks=true,linkcolor=Red,citecolor=Green,urlcolor=Blue]{hyperref}
\usepackage[super]{nth}
\usepackage[open, openlevel=2, depth=3, atend]{bookmark}
\hypersetup{pdfstartview=XYZ}
\usepackage{dsfont}
\usepackage{epigraph}
\def\?[#1]{\textbf{[#1]}\marginpar{\Large{\textbf{??}}}}

\numberwithin{equation}{section}

\newtheorem{theorem}{Theorem}[section]
\newtheorem{lemma}[theorem]{Lemma}
\newtheorem{proposition}[theorem]{Proposition}
\newtheorem{corollary}[theorem]{Corollary}

\newtheorem{remark}[theorem]{Remark}

\newcommand{\mc}{\mathcal}

\newcommand{\norm}[1]{\left\Vert#1\right\Vert}

\newcommand{\bbar}{\overline}
\renewcommand{\det}{\text{det}'}

\let\Im=\Imag

\let\Re=\Real

\DeclareMathOperator{\Vol}{Vol}

\renewcommand{\tilde}{\widetilde}          
\DeclareMathSymbol{\leqslant}{\mathalpha}{AMSa}{"36} 
\DeclareMathSymbol{\geqslant}{\mathalpha}{AMSa}{"3E} 
\DeclareMathSymbol{\eset}{\mathalpha}{AMSb}{"3F}     
\renewcommand{\leq}{\;\leqslant\;}                   
\renewcommand{\geq}{\;\geqslant\;}                   

\newcommand{\C}{\mathbb{C}}

\newcommand{\R}{\mathbb{R}}
\newcommand{\Z}{\mathbb{Z}}
\renewcommand{\H}{\mathbb{H}}

\newcommand{\Q}{\mathbb{Q}}
\newcommand{\U}{\mathbb{U}}
\newcommand{\E}{\mathbb{E}}
\renewcommand{\P}{\mathbb{P}}
\newcommand{\ind}{\mathds{1}}

\def\S{\mathbb{S}}
\def\T{\mathbb{T}}

\def\bi{\begin{itemize}}
\def\ei{\end{itemize}}

\def\bnum{\begin{enumerate}}
\def\enum{\end{enumerate}}
\def\<#1{\langle #1 \rangle}

\def\g{\mathbf{g}}
\def\U{\Upsilon_\frac{\gamma}{2}}

\title[Modular bootstrap from the path integral]{Modular bootstrap agrees with path integral in the large moduli limit}

\author[Guillaume Baverez]{Guillaume Baverez}
\thanks{Supported by the EPSRC grant EP/L016516/1 for the University of Cambridge CDT (CCA)}
 \address{University of Cambridge, Centre for Mathematical Sciences, Cambridge CB3 0WA, UK}
 \email{gb539@cam.ac.uk}

\begin{document}

\begin{abstract}
Based on the rigorous path integral formulation of Liouville Conformal Field Theory initiated by David-Kupiainen-Rhodes-Vargas \cite{DKRV1} on the Riemann sphere and David-Rhodes-Vargas \cite{DRV} on the torus of modulus $\tau$, we give the exact asymptotic behaviour of the 1-point toric correlation function as $\Im\tau\to\infty$.

In agreement with formulae predicted within the bootstrap formalism of theoretical physics, our results feature an $(\Im\tau)^{-3/2}$ decay rate and we identify the derivative of DOZZ formula in the limit.
\end{abstract}
\maketitle




  



 
\section{Introduction}
\label{sec:intro}
In theoretical physics, there are two approaches to Conformal Field Theories (CFTs). The first is the Hamiltonian approach: it consists in quantising an action functional and is usually treated with Feynman path integrals. The second is the conformal bootstrap: an abstract machinery used to classify CFTs from the algebraic information encoded by conformal invariance. Liouville CFT arises in the Hamiltonian approach in many fields of theoretical physics, notably in string theory \cite{Po,Da,DhPh}. In the conformal bootstrap, it is the first CFT with continuous spectrum that physicists were able to ``solve" \cite{Ri}.

From a mathematical point of view, path integrals are not rigorous, but recently, a rigorous probabilistic framework based on the Gaussian Free Field (GFF) and Gaussian Multiplicative Chaos (GMC) was introduced in order to make sense of the path integral approach to LCFT on any compact Riemann surface \cite{DKRV1,DRV,GRV}. The remaining challenge for probabilists is to show that the path integral carries all the representation theoretic aspects predicted by the conformal bootstrap.

A first step was made in this direction when \cite{KRV2} showed that the structure constants of LCFT (see Section \ref{subsec:bootstrap}) satisfy the so-called DOZZ formula. The term ``bootstrapping" refers to the recursive computation of correlation functions from the structure constants, and this paper checks the validity of this recursion in a weakly interacting regime. From a probabilistic point of view, the DOZZ formula is a highly non-trivial integrability result on GMC, and it was soon followed by the results of \cite{Remy,remy_zhu} where similar methods were implemented in order to compute the law of GMC on the unit circle and interval.

%

 	\subsection{Path integral}
 	\label{subsec:path_integral}
 Let $M$ be either the Riemann sphere $\S^2\simeq\hat{\C}=\C\cup\{\infty\}$ or the torus $\T_\tau\simeq\C/(\Z+\tau\Z)$ for some $\tau\in\H:=\{\Im\tau>0\}$. The \emph{Liouville action} with background metric $g$ on $M$ is the map $S_\mathrm{L}:\Sigma\to\R$ (where $\Sigma$ is some function space to be determined) defined by\footnote{Usually the Liouville action features an additional curvature term. We omitted it since we will work with a background metric which is flat everywhere except on the unit circle.}
 \begin{equation}
 \label{eq:def_liouville}
 S_{\mathrm{L}}(X;g)=\frac{1}{4\pi}\int_M\left(|\nabla X|^2+4\pi\mu e^{\gamma X}g(z)\right)dz,
\end{equation} 	
where $\mu>0$ is the cosmological constant (whose value is unimportant for this paper) and $\gamma\in(0,2)$ is the parameter of the theory. Liouville quantum field theory is the measure formally defined by
\begin{equation}
\label{eq:def_liouville_field}
\langle F\rangle:=\int F(X)e^{-S_{\mathrm{L}}(X;g)}DX
\end{equation}
for all continuous functional $F$. Here, $DX$ should stand for ``Lebesgue" measure on $\Sigma$. Of course, this does not make sense mathematically but it is possible to interpret the formal measure 
\begin{equation}
\label{eq:gff_measure}
\frac{1}{Z_\mathrm{GFF}}e^{-\frac{1}{4\pi}\int_M|\nabla X|^2dz}DX
\end{equation} 
as a Gaussian probability measure on some Hilbert space (to be determined). The resulting field is called the Gaussian Free Field and the quantity $Z_\mathrm{GFF}$ is a ``normalising constant" turning the measure \eqref{eq:gff_measure} into a probability measure. We will refer to it as the partition function of the GFF (see Section \ref{subsubsec:gff}).

As it turns out, the GFF does not live in the space of continuous functions (not even in $L^2$) but is rather a distribution in the sense of Schwartz. It can be shown that the GFF almost surely lives in the topological dual of the Sobolev space $H^1$ with respect to the $L^2$ product. Hence the exponential term $e^{\gamma X}dz$ appearing in the action is not \textit{a priori} well defined, but it can be made sense of after a regularising procedure based on Kahane's theory of Gaussian Mutiplicative Chaos (GMC) (see Section \ref{subsubsec:gmc}).

The main observables of the theory are the \emph{vertex operators} $V_\alpha(z_0)=e^{\alpha X(z_0)}$ for any $z_0\in M$ and $\alpha<Q:=\frac{2}{\gamma}+\frac{\gamma}{2}$. The point $z_0$ is called an \emph{insertion} as it has the interpretation of puncturing $M$ with a conical singularity of order $\alpha/Q$ (see \cite{HMW} and Appendix \ref{app:conical}). The coefficient $\alpha$ is called the \emph{Liouville momentum} and $\Delta_\alpha:=\frac{\alpha}{2}(Q-\frac{\alpha}{2})$ is called the \emph{conformal dimension}. The vertex operators give rise to the correlation functions $\langle\prod_{n=1}^NV_{\alpha_n}(z_n)\rangle$ which are defined for any pairwise disjoint $z_1,...,z_N\in M$ and $\alpha_1,...,\alpha_N\in\R$ satisfying the so-called \emph{Seiberg bounds}
\begin{equation}
\label{eq:seiberg_bounds}
\sum_{n=1}^N\frac{\alpha_n}{Q}-\chi(M)>0\qquad\qquad\forall n,\alpha_n<Q,
\end{equation}
where $\chi(M)$ is the Euler characteristic. The Seiberg bounds have a geometric nature: the $\alpha_n/Q$ singularity introduced by $V_{\alpha_n}(z_n)$ is integrable only if $\alpha_n<Q$, hence the second bound in \eqref{eq:seiberg_bounds}. On the other hand, Gauss-Bonnet theorem shows that the first bound is equivalent to asking for the total curvature on the surface $M\setminus\{z_1,...,z_N\}$ with prescribed conical singularities $\alpha_n/Q$ at $z_n$ to be negative. In particular, the correlation function exists only if $N\geq3$ for the sphere and $N\geq1$ for the torus.

We now briefly recall the results that will be needed in order to state the main result. Consider the Riemann sphere $\S^2\simeq\hat{\C}$ equipped with the metric $g(z)=|z|_+^{-4}$ (with the notation $|z|_+=\max(1,|z|)$). We will refer to this metric as the \emph{cr\^epe metric} as it consists in two flat disks glued together (as can be seen from the change of variable $z\mapsto1/z$). The 3-point function enjoys some conformal covariance under M\"obius transformations \cite{DKRV1}, implying that we can choose to put the insertions at $0,1,\infty$. It was shown in \cite{KRV2} that for all $\alpha_1,\alpha_2,\alpha_3$ satisfying the Seiberg bounds, $\langle V_{\alpha_1}(0)V_{\alpha_2}(1)V_{\alpha_3}(\infty)\rangle_{\S^2}=C_\gamma(\alpha_1,\alpha_2,\alpha_3)$ where $C_\gamma(\alpha_1,\alpha_2,\alpha_3)$ is the celebrated DOZZ formula (see Appendix \ref{app:dozz}).

Recall that a torus is a curve $\C/(\Z+\tau\Z)$ with $\tau\in\H$. The \emph{moduli group} $\Gamma=\mathrm{PSL}(2,\Z)$ acts on $\H$ via linear fractional transformation
\[\psi.\tau=\frac{a\tau+b}{c\tau+d}\]
for all $\psi=\begin{pmatrix}
a & b \\ 
c & d
\end{pmatrix} \in\Gamma$. The \emph{moduli space} is the quotient $\mc{M}:=\Gamma\setminus\H$. Two tori with moduli $\tau,\tau'$ respectively are conformally equivalent if and only if there exists $\psi\in\Gamma$ such that $\tau'=\psi.\tau$. The fundamental domain of $\mc{M}$ is the set $\{z\in\H,\;\Re(z)\in(-1/2,1/2]\text{ and }|z|>1\}\cup\{e^{i\theta},\theta\in[\frac{\pi}{3},\frac{\pi}{2}]\}$ (see Figure \ref{fig:modular_curve}), so that the boundary of the moduli space can be approached by moduli $\tau=\frac{it}{\pi}$ for large $t$. These correspond to ``skinny" tori. From \cite{DRV} it is possible to define the 1-point correlation function $\langle V_\alpha(0)\rangle_\tau$ with flat backgroud metric for each modulus $\tau\in\mc{M}$ and $\alpha\in(0,Q)$,

\begin{figure}[h!] 
\centering
\includegraphics[scale=0.8]{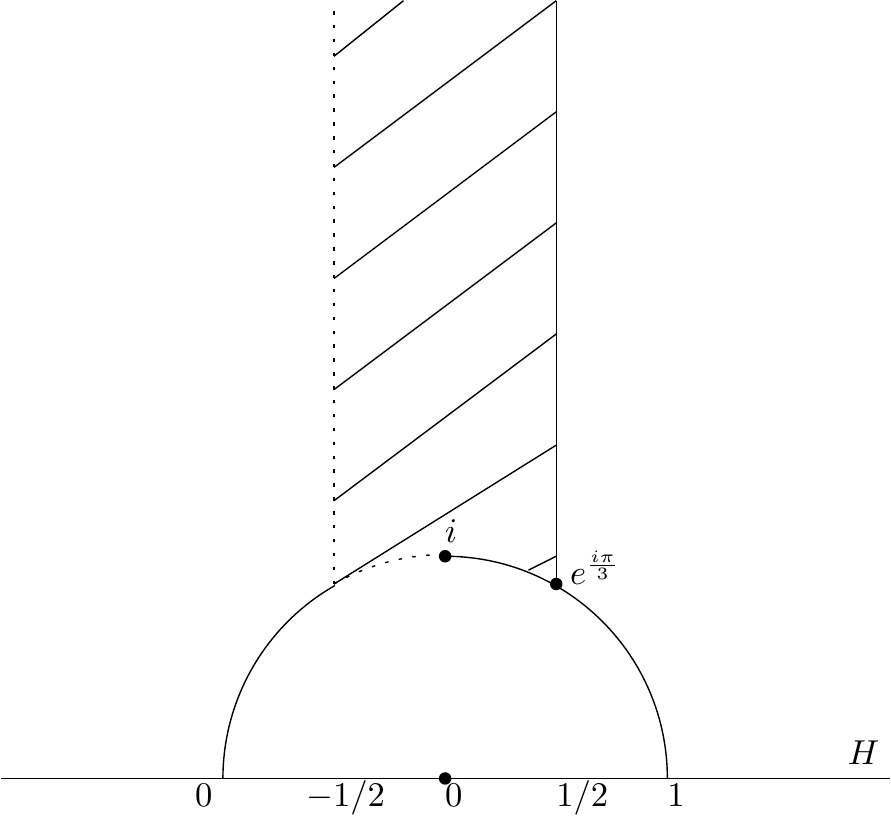}
\caption{\label{fig:modular_curve}The moduli space $\mc{M}=\Gamma\setminus\H$ (hashed). The vertical lines are identified, so that it is topologically a sphere with three marked points at $e^{i\pi/3},i$ and $\infty$. The interesting boundary point is $\infty$, and we will approach it using moduli $\tau=\frac{it}{\pi}$ for large $t$. These correspond to ``skinny" tori.}
\end{figure}

Using the framework of CFT known as the conformal bootstrap, phycisists have argued that all correlation functions on any surface can be derived from the three-point function on the sphere by some topological recursion (see Section \ref{subsec:bootstrap}). In the case of the one-point function on the torus, the formula involves an integral over some algebraically defined objects that do not yet have a probabilistic representation (see Equation \eqref{eq:torus_bootstrap}). However, these objects have nice asymptotic behaviours as $\Im\tau\to\infty$, explaining why we were able to compute the asymptotic behaviour of the one-point toric function and match it with the bootstrap prediction in this limit.

 	
 	\subsection{Conformal bootstrap}
 	\label{subsec:bootstrap}
  From the operator theoretic perspective, a quantum field theory is the data of a self-adjoint non-negative Hamiltonian acting on some Hilbert space. In their founding paper, \cite{BPZ} argued that the Hilbert space of a $2d$ conformal field theory must carry a representation of the Virasoro algebra. This strong constraint on the structure of the Hilbert space led to spectacular integrability results, among which the DOZZ formula from Liouville theory. Although the representation theory of the Virasoro algebra is well-understood mathematically, it is only a conjecture that the path integral of the quantised Liouville action carries the expected algebraic structure. Thus, except for the results of \cite{KRV1,KRV2}, all the formulae from the conformal bootstrap are to be considered as predictions and not facts. 
 
 In the conformal bootstrap framework, any CFT should be characterised by
 \begin{enumerate}
 \item The \emph{spectrum} of the Hamiltonian $S\subset\R_+$. For each $\alpha\in\C$ such that $\Delta_\alpha\in S$, the field $V_\alpha(\cdot)$ is called a \emph{primary field}. It is important to note that the conformal bootstrap assumes that vertex operators are defined for all $\alpha\in\C$ and not necessarily for $\alpha$ in the ``physical region" defined by the Seiberg bounds. The spectrum of Liouville theory is conjectured to be $[\frac{Q^2}{4},\infty)$, corresponding to momenta $\alpha\in Q+i\R$.
 \item The \emph{structure constants}, i.e. the three-point functions on the sphere $\langle V_{\alpha_1}(0)V_{\alpha_2}(1)V_{\alpha_3}(\infty)\rangle_{\S^2}$. In Liouville CFT, the structure constants are given by the DOZZ formula $C_\gamma(\alpha_1,\alpha_2,\alpha_3)$ \cite{DO,ZZ}. Correlation functions are meromorphic functions of each $\alpha\in\C$.
 \end{enumerate}
 From the data of the spectrum and the structure constants, the bootstrap machinery gives a way to compute recursively all correlation functions on any Riemann surface of any genus. Thus, ``solving" a theory means finding both the spectrum and the structure constants.
 
 The two most simple examples are the 4-point spherical and the 1-point toric correlation function. Given two copies $M_1,M_2$ of the thrice punctured sphere $\S^2\setminus\{0,1,\infty\}$, one can glue together annuli neighbourhoods of punctures in $M_1$ and $M_2$ to produce a 4-punctured sphere (see Figure \ref{fig:glue}). Similarly, given one instance of the thrice-punctured sphere, one can glue together annuli neighbourhoods of $0$ and $\infty$ to produce the once-punctured torus.
 
 \begin{figure}[h!]
 \centering
 \includegraphics[scale=0.6]{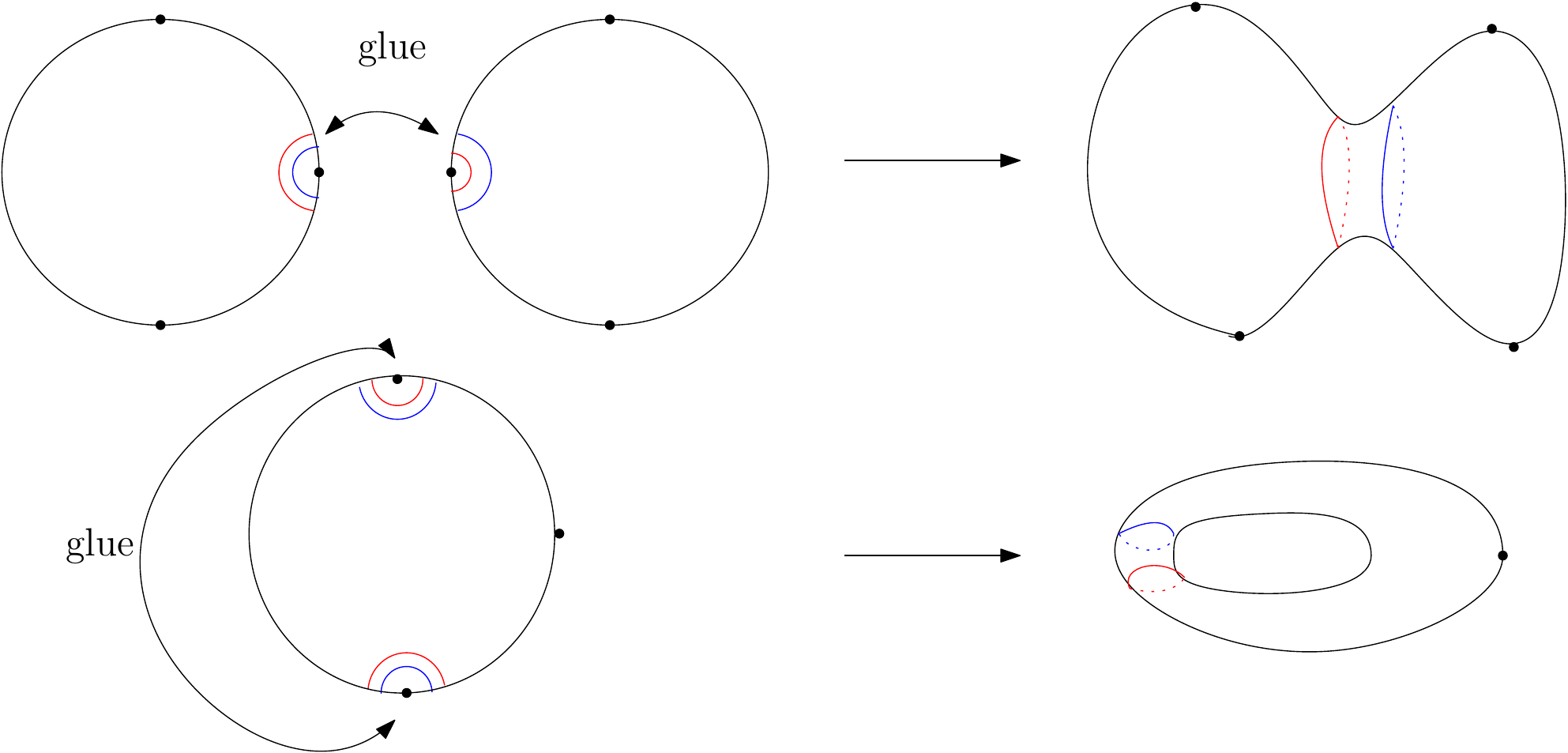}
 \caption{\label{fig:glue}\textbf{Top:} On the left, two instances of the thrice-punctured sphere with annuli neighbourhoods to be identified (curves of the same colour are identified). The resulting surface on the right: a sphere with 4 marked points. \textbf{Bottom:} Annuli neighbourhoods of the north and south pole are identified to produce a torus with one marked point.}
 \end{figure}
 
 More generally, this procedure gives a way to construct any Riemann surface of genus $\g_1+\g_2$ and $n_1+n_2$ punctures by gluing a surface of genus $\g_1$ and $n_1+1$ punctures to a surface of genus $\g_2$ and $n_2+1$ punctures (see \cite{TV} for details of this construction). Similarly a surface of genus $\g$ and $n+2$ punctures gives a surface of genus $\g+1$ and $n+2$ punctures after gluing together two distinct punctured neighbourhoods. This gives a recursive procedure to construct any Riemann surface using only instances of the thrice-punctured sphere. This construction is one of the driving ideas behind the fact that three-point functions are building blocks for CFTs.
 
 The two simple examples above are the starting point of the bootstrap programme as they require only one gluing. Physicists have predicted formulae -- called the bootstrap equations -- that compute these correlation functions using the spectrum and the structure constants. The bootstrap equation for the 4-point function on the sphere is given by\footnote{We add the superscript $^{\text{cb}}$ for ``conformal bootstrap" and to differentiate it from the path integral correlation function.}\cite{BZ}
 \begin{equation}
 \label{eq:4_point_bootstrap}
 \begin{aligned}
\langle V_{\alpha_1}(0)V_{\alpha_2}(z)&V_{\alpha_3}(1)V_{\alpha_4}(\infty)\rangle_{\S^2}^{\text{cb}}=\frac{1}{8\pi}|z|^{2(\frac{Q^2}{4}-\Delta_1-\Delta_2)}\\
&\times\int_{-\infty}^\infty|z|^{2P^2}C_\gamma(\alpha_1,\alpha_2,Q-iP)C_\gamma(Q+iP,\alpha_3,\alpha_4)|\mc{F}_P^{\alpha_{1234}}(z)|^2dP,
\end{aligned}
 \end{equation}
 where $\Delta_i=\frac{\alpha_i}{2}(Q-\frac{\alpha_i}{2})$ ($i=1,2$) and $\mc{F}_P^{\alpha_{1234}}(z)=1+o(1)$ is the so-called \emph{Virasoro conformal block} -- a holomorphic function of $z$, universal in the sense that it only depends on the $\alpha_i$'s, $P$ and $\gamma$. 

 There is a similar formula to \eqref{eq:def_1_point} for the 1-point toric function \cite[Equation (20)]{HJS}, which is the one this paper is concerned about. For a torus of modulus $\tau$, we have
 \begin{equation}
 \label{eq:torus_bootstrap}
 \langle V_\alpha(0)\rangle_\tau^{\text{cb}}=\frac{1}{2}\int_\R C_\gamma(Q-iP,\alpha,Q+iP)\left|q^{\frac{P^2}{4}}\eta(q)^{-1}\mc{H}_{\gamma,P}^\alpha(q)\right|^2dP,
 \end{equation}
where $q=e^{2i\pi\tau}$ is the \emph{nome} and $\eta(\cdot)$ is \emph{Dedekind's \^eta function}. Here the so-called \emph{elliptic conformal bloc} $\mc{H}_{\gamma,P}^{\alpha}$ admits a power series expansion in $q$
\[\mc{H}_{\gamma,P}^\alpha(q)=\frac{\eta(q)}{q^{1/24}}\left(1+\sum_{n=1}^\infty H_{\gamma,P}^{\alpha,n}q^n\right)\]
and the function in the brackets is holomorphic in $q$. The elliptic conformal blocks should be understood as a basis of solutions for the one-point toric function, and they are universal in the sense that they depend only on $\alpha,\gamma$ and $P$. We will refer to equation \eqref{eq:torus_bootstrap} as the \emph{modular bootstrap}. Again, this formula should be valid \textit{a priori} only for a primary field but we will show that it is true for $\alpha\in(0,Q)$ in the path integral framework when $\Im\tau\to\infty$.
 
At this stage, let us stress again that equations \eqref{eq:4_point_bootstrap} and \eqref{eq:torus_bootstrap} should be understood only as guesses since there is still no mathematical justification for them. In general, one way to establish rigorously the validity of the conformal bootstrap would be to recover its results using the rigorous path integral approach of DKRV. This is usually a hard matter but some works were made in this direction \cite{KRV1,KRV2}. In the first paper, the authors showed the validity of some aspects of the bootstrap approach -- namely BPZ equation and Ward identities --, while the second is a proof of the DOZZ formula.

From the point of view of probability, both the conformal blocks and the spectrum are not understood (there is not even a probabilistic interpretation of complex Liouville momenta). As we mentioned earlier, the integral in \eqref{eq:torus_bootstrap} simplifies as $\Im\tau\to\infty$, namely the conformal blocks tend to 1 and the integral freezes at $P=0$, avoiding to deal with complex insertions.

 	\subsection{Main result and outline}
 	\label{subsec:result}
Suppose $\tau=\frac{it}{\pi}$ with $t>0$ large, so that $q=e^{-2t}$ is real and small. Recall that the DOZZ formula is meromorphic and symmetric with respect to the real axis, hence 
\[C_\gamma(Q+iP,\alpha,Q-iP)\underset{P\to0}{\sim}P^2\partial^2_{\alpha_1\alpha_3}C_\gamma(Q,\alpha,Q).\]
Taking $\mc{H}_{\gamma,P}^\alpha(q)\equiv1$ uniformly in $q$ as $P\to0$, equation \eqref{eq:torus_bootstrap} gives in the limit $t\to\infty$
\begin{equation}
\label{eq:bootstrap_limit}
\begin{aligned}
\langle V_\alpha(0)\rangle_{\frac{it}{\pi}}^{\text{cb}}
&=\frac{|\eta(\frac{it}{\pi})|^{-2}}{2}\int_\R C_\gamma(Q-iP,\alpha,Q+iP)\left|q^{\frac{P^2}{4}}\mc{H}^\alpha_{\gamma,P}(q)\right|^2dP\\
&\sim\frac{|\eta(\frac{it}{\pi})|^{-2}}{2}\int_\R C_\gamma(Q-iP,\alpha,Q+iP)e^{-\frac{tP^2}{2}}\left|\mc{H}_{\gamma,P}^\alpha(q)\right|^2dP\\
&=\frac{|\eta(\frac{it}{\pi})|^{-2}}{2}t^{-1/2}\int_\R C_\gamma\left(Q-i\frac{P}{\sqrt{t}},\alpha,Q+i\frac{P}{\sqrt{t}}\right)e^{-\frac{P^2}{2}}\left|\mc{H}_{\gamma,\frac{P}{\sqrt{t}}}^\alpha(q)\right|^2dP\\
&\sim\frac{|\eta(\frac{it}{\pi})|^{-2}}{2} t^{-3/2}\partial^2_{\alpha_1\alpha_3}C_\gamma(Q,\alpha,Q)\int_\R P^2e^{-\frac{P^2}{2}}dP\\
&\sim\sqrt{\frac{\pi}{2}}\left|\eta\left(\frac{it}{\pi}\right)\right|^{-2}t^{-3/2}\partial^2_{\alpha_1\alpha_3}C_\gamma(Q,\alpha,Q).
\end{aligned}
\end{equation}
Rewriting this in terms of the modulus, we have in the limit $\Im\tau\to\infty$
\begin{equation}
\label{eq:lim_correl}
\langle V_\alpha(0)\rangle_\tau^{\text{cb}}\sim\frac{\sqrt{2}}{\pi}|\eta(\tau)|^{-2}(\Im\tau)^{-3/2}\partial_{\alpha_1\alpha_3}^2C_\gamma(Q,\alpha,Q).
\end{equation}
 
 There are two noticeable facts about the asymptotic behaviour of $\langle V_\alpha(0)\rangle_{\frac{it}{\pi}}$:
 \begin{itemize}
 \item There is a polynomial decay in $t^{-3/2}$ correcting the exponential term $|\eta(\frac{it}{\pi})|^{-2}$.
 \item The limit is expressed using the derivative of the DOZZ formula at the critical points $\alpha_1=\alpha_3=Q$.
 \end{itemize}

Throughout, we will write $\T_t$ for a torus with modulus $\tau=\frac{it}{\pi}$ and think of $t$ large. Our representation for $\T_t$ is the rectangle $\T_t:=(-t,t]\times\S^1$ with edges $\{-t\}\times\S^1$ and $\{t\}\times\S^1$ identified, and equipped with the flat metric. The reason for this choice of parametrisation is that the variable $t$ will appear as the time driving a Brownian motion. 

Let $\mc{C}_\infty:=\R\times\S^1$ be the infinite cylinder. This surface is the image of the twice-punctured sphere $\hat{\C}\setminus\{0,\infty\}$ under the change of coordinates $\psi:\mc{C}_\infty\to\hat{\C}\setminus\{0,\infty\},z\mapsto e^{-z}$. In the sequel, we will always parametrise the sphere with these coordinates. Of particular interest for us will be the correlation function $\langle V_\lambda(0)V_\alpha(1)V_\lambda(\infty)\rangle_{\S^2}$ for $\lambda,\alpha\in(0,Q)$ and $\sigma=2(\lambda-Q)+\alpha>0$. In the cylinder coordinates, these take the form \cite{KRV2}
\begin{equation}
\label{eq:def_3_point}
\langle V_\lambda(0)V_\alpha(1)V_\lambda(\infty)\rangle_{\S^2}=2\gamma^{-1}\mu^{-\frac{Q\sigma}{\gamma}}\Gamma\left(\frac{Q\sigma}{\gamma}\right)\E\left[\left(\int_{\mc{C}_\infty}e^{\gamma((\lambda-Q)|t|+\alpha G(0,t+i\theta))}dM^\gamma(t,\theta)\right)^{-\frac{Q\sigma}{\gamma}}\right],
\end{equation}
where $G$ is Green's function on $\mc{C}_\infty$ with zero average on $\{0\}\times\S^1$ and $M^\gamma$ is the chaos measure associated to a GFF on $\mc{C}_\infty$.

The negative drift $\lambda-Q$ is essential in order to make the total GMC mass finite near $\pm\infty$. On the contrary if $\lambda=Q$, the GMC mass is a.s. infinite and the correlation function vanishes. In this critical case, we consider the truncated correlation function
\begin{equation}
\label{eq:truncated}
\langle V_Q(0)V_\alpha(1)V_Q(\infty)\rangle_t=2\gamma^{-1}\mu^{-\frac{\alpha}{\gamma}}\Gamma\left(\frac{\alpha}{\gamma}\right)\E\left[\left(\int_{\mc{C}_t}e^{\gamma\alpha G(0,\cdot)}dM^\gamma(s,\theta)\right)^{-\frac{\alpha}{\gamma}}\right],
\end{equation}
where $\mc{C}_t:=(-t,t)\times\S^1$.

The truncated correlation function is just the correlation function where we integrate the GMC measure outside a small disks of radius $e^{-t}$ away from the singularities (when seen in the planar coordinates).

As for the torus $\T_t$, the 1-point function is defined by 
\begin{equation}
\label{eq:def_1_point}
\langle V_\alpha(0)\rangle_{\frac{it}{\pi}}:=2\gamma^{-1}\mu^{-\frac{\alpha}{\gamma}}\Gamma\left(\frac{\alpha}{\gamma}\right)\left(\frac{t}{\pi}\right)^{-1/2}|\eta(\frac{it}{\pi})|^{-2}\E\left[\left(\int_{\T_t}e^{\gamma\alpha G_t(0,\cdot)}dM^\gamma_t\right)^{-\frac{\alpha}{\gamma}}\right],
\end{equation}
where $G_t$ is Green's function on $\T_t$.

Our main result, stated as Theorem \ref{thm:result} below shows that we recover the same polynomial rate and the derivative DOZZ formula when working with the correlation function computed in the path integral framework.

\begin{theorem}
\label{thm:result}
Let $\langle V_\alpha(0)\rangle_{\frac{it}{\pi}}$ be the 1-point toric correlation function given by \eqref{eq:def_1_point}. Then
\begin{equation}
\label{eq:result}
\langle V_\alpha(0)\rangle_\frac{it}{\pi}\underset{t\to\infty}{\sim}\frac{3}{4\sqrt{\pi}}\left|\eta(\frac{it}{\pi})\right|^{-2}t^{-3/2}\partial^2_{\alpha_1\alpha_3}C_\gamma(Q,\alpha,Q).
\end{equation}
\end{theorem}

\begin{corollary}
\label{cor:result}
In the setting of Theorem 1.1, we have for $\tau\in\mc{M}$:
\[\langle V_\alpha(0)\rangle_\tau\underset{\Im\tau\to\infty}{\sim}\frac{3}{4\pi^2}|\eta(\tau)|^{-2}(\Im\tau)^{-3/2}\partial^2_{\alpha_1\alpha_3}C_\gamma(Q,\alpha,Q).\]
\end{corollary}

\begin{remark}
The fact that we don't recover the same global overall factor as in equation \eqref{eq:lim_correl} is irrelevant since the correlation functions are defined up to multiplicative factor.
\end{remark}

	\subsection{Steps of the proof}
	\label{subsec:steps_of_proof}
There will be two steps in the proof of Theorem \ref{thm:result}. First we will compute the exact asymptotic behaviour of $\langle V_Q(0)V_\alpha(1)V_Q(\infty)\rangle_t$ as $t\to\infty$ (Proposition \ref{prop:spheric_limit}) and second we will compare $\langle V_\alpha(0)\rangle_{\frac{it}{\pi}}$ to $\langle V_Q(0)V_\alpha(1)V_Q(\infty)\rangle_t$ (Proposition \ref{prop:spheric_equiv_toric}). This is the point of using the cylinder coordinates for the sphere, as we can embed $\T_t$ into $\mc{C}_t$. Namely, we will show that negative moments of GMC on $\T_t$ and on $\mc{C}_t$ have the same asymptotic behaviour, up to some explicit constant.

\begin{proposition}
\label{prop:spheric_limit}
For all $\alpha\in(0,Q)$, 
\begin{equation}
\underset{t\to\infty}{\lim}t\langle V_Q(0)V_\alpha(1)V_Q(\infty)\rangle_t=\frac{1}{2\pi}\partial^2_{\alpha_1\alpha_3}C_\gamma(Q,\alpha,Q).
\end{equation}
\end{proposition}

\begin{proposition}
\label{prop:spheric_equiv_toric}
Let $X$ be a GFF on $\mc{C}_\infty$ and $X_t$ be a GFF on $\T_t$, i.e. $X$ and $X_t$ have respective covariances $G$ and $G_t$ (Green's function with zero average on $\{0\}\times\S^1$). Let $dM^\gamma$ and $dM^\gamma_t$ be the associated chaos measures. Then for all $r>0$ and $\alpha\in(0,Q)$,
\begin{equation}
\label{eq:chaos_limit}
\underset{t\to\infty}{\lim}t\E\left[\left(\int_{\T_t}e^{\gamma\alpha G_t(0,z)}dM_t^\gamma(z)\right)^{-r}\right]=\frac{3}{2}\underset{t\to\infty}{\lim}t\E\left[\left(\int_{\mc{C}_t}e^{\gamma\alpha G(0,z)}dM^\gamma(z)\right)^{-r}\right].
\end{equation}
\end{proposition}
%
%
We will prove these propositions in Section \ref{sec:proof}. For now, we use Propositions \ref{prop:spheric_limit} and \ref{prop:spheric_equiv_toric} to prove Theorem \ref{thm:result} and Corollary \ref{cor:result}.

\begin{proof}[Proof of Theorem \ref{thm:result}]
Using Propositions \ref{prop:spheric_limit} and \ref{prop:spheric_equiv_toric}, we have
\begin{equation}
\begin{aligned}
\langle V_\alpha(0)\rangle_{\frac{it}{\pi}}
&=\sqrt{\pi}\frac{2\gamma^{-1}\mu^{-\frac{\alpha}{\gamma}}\Gamma\left(\frac{\alpha}{\gamma}\right)}{t^{1/2}|\eta(\frac{it}{\pi})|^2}\E\left[\left(\int_{\T_t}e^{\gamma\alpha G_t(0,\cdot)}dM^\gamma_t\right)^{-\frac{\alpha}{\gamma}}\right]\\
&\underset{t\to\infty}{\sim}\frac{3\sqrt{\pi}}{2}t^{-1/2}|\eta(\frac{it}{\pi})|^{-2}\langle V_Q(0)V_\alpha(1)V_Q(\infty)\rangle_t\\
&\underset{t\to\infty}{\sim}\frac{3}{4\sqrt{\pi}}t^{-3/2}|\eta(\frac{it}{\pi})|^{-2}\partial^2_{\alpha_1\alpha_3}C_\gamma(Q,\alpha,Q).
\end{aligned}
\end{equation}
In particular we recover the asymptotic formula of equation \eqref{eq:bootstrap_limit} up to an explicit global multiplicative constant.

\end{proof}

\begin{proof}[Proof of Corollary \ref{cor:result}]
In this proof and this proof only, we make change the embedding and embed all tori in the square $[0,1]^2$ as in \cite{DRV}. We only need to compare the negative moments of GMC for tori with moduli $\tau,\tau'$ such that $\Im\tau=\Im\tau'$ and show that they have the same asymptotic behaviour as $\Im\tau\to\infty$. 

Let $\tau\in\mc{M}$ with $\Im\tau=\frac{t}{\pi}$.  Let $G_\tau$ be Green's function on the torus $\T_\tau$ of modulus $\tau$ and set $G_\tau(x):=G_\tau(0,x)$. It is readily seen from \cite[Equation (3.4)]{DRV} that
\[|G_\tau(x)-G_\frac{it}{\pi}(x)|=O(e^{-2t})\]
uniformly in $x\in\T_\tau$. Now let $dM^\gamma_\tau$ and $dM^\gamma_\frac{it}{\pi}$ be the GMC measures of a GFF on $\T_\tau$ and $\T_\frac{it}{\pi}$ respectively. By Kahane's convexity inequality (see Section \ref{subsubsec:gmc}) we have for all $r>0$
\begin{equation}
\E\left[\left(\int_{\T_\tau}e^{\gamma\alpha G_\tau(0,\cdot)}dM^\gamma_\tau\right)^{-r}\right]=\E\left[\left(\int_{\T_\frac{it}{\pi}}e^{\gamma\alpha G_\frac{it}{\pi}(0,\cdot)}dM^\gamma_{\frac{it}{\pi}}\right)^{-r}\right](1+O(e^{-2t})).
\end{equation}
This concludes the proof.
\end{proof}

The rest of this paper is devoted to proving Propositions \ref{prop:spheric_limit}, \ref{prop:spheric_equiv_toric}. This will be done is Section \ref{sec:proof} while Section \ref{sec:background} gives the necessary probabilistic background needed for the proofs.

\bigskip 

\paragraph*{\textbf{Acknowledgements:}} 
The author is grateful R\'emi Rhodes for bringing him to the study LCFT near the boundary of moduli spaces, and to Vincent Vargas for telling him about the modular bootstrap equation. We thank both of them for many interesting discussions. We also thank the anonymous referee for their careful and patient reading of the manuscript.

\section{Background}
\label{sec:background}
In this section, we recall the definitions of the basic objects needed to define the correlation functions \eqref{eq:def_3_point} and \eqref{eq:def_1_point} (namely the GFF and GMC) and we give a derivation of the expression of these correlation functions.
	\subsection{Gaussian Free Field}
	\label{subsubsec:gff}
We give a basic introduction to the Gaussian Free Field (GFF) on the complete cylinder $\mc{C}_\infty$ and the torus $\T_t$ (we refer the reader to \cite{Du,DMS,DKRV1,DRV}). 

To begin with, let us consider the case of $\mc{C}_\infty$ endowed with the flat metric. Let $H^1_0(\mc{C}_\infty)$ be the set of functions $f:\mc{C}_\infty\to\R$ with weak derivative in $L^2(\mc{C}_\infty)$ and such that $\int_0^{2\pi}f(0,\theta)d\theta=0$. Then the (non-negative) Laplacian $-\frac{1}{2\pi}\Delta$ has a well defined inverse $G:L^2(\mc{C}_\infty)\to H^1_0(\mc{C}_\infty)$ called the \emph{Green function}. It has a kernel satisfying for all $x\in\mc{C}_\infty$
\begin{equation}
\left\lbrace
\begin{aligned}
&-\frac{1}{2\pi}\Delta G(x,\cdot)=\delta_x\\
&\int_0^{2\pi}G(x,i\theta)d\theta=0.
\end{aligned}
\right.
\end{equation} 
The GFF on $\mc{C}_\infty$ is the Gaussian field $X$ on whose covariance kernel is given by Green's function
\[\E[X(x)X(y)]=G(x,y).\]
This is done at the formal level, since Green's function blows up logarithmically near the diagonal. However, it is possible to show that such a field exists and that it almost surely lives in $H^{-1}_0(\mc{C}_\infty)$. Hence the GFF on $\mc{C}_\infty$ is a distribution on $\mc{C}_\infty$ (and not a function).

We can define $H^1_0(\T_t)$ similarly as the space of functions $f:\T_t\to\R$ with weak derivatives in $L^2(\T_t)$ and vanishing mean on $\T_t$. The Laplacian $-\frac{1}{2\pi}\Delta_t$ on $\T_t$ also has a Green's function $G_t:L^2(\T_t)\to H^1_0(\T_t)$. 
		
As explained in Section \ref{subsec:path_integral}, the formal measure $e^{-\frac{1}{4\pi}\int_M|\nabla X|^2}DX$ should be interpreted as a Gaussian measure. To fix ideas, let us consider the case of the torus $\T_t$. Then the map 
\[(f,g)\mapsto-\frac{1}{2\pi}\int_{\T_t}\Delta_t f\cdot g=:\langle f,g\rangle_{\nabla}\] 
defines an inner-product on $H^1_0(\T_t)$ that we call the \emph{Dirichlet energy}. We write $\norm{\cdot}_{\nabla}$ for the associated norm. By analogy with the finite dimensional case, we want to interpret the density $e^{-\frac{1}{2}\norm{X}_{\nabla}^2}DX$ as that of a centred Gaussian random variable with covariance kernel given by the inverse of $-\frac{1}{2\pi}\Delta$, i.e. Green's function $G_t$. This is nothing but the GFF of the previous paragraph. To keep with the analogy with the finite dimensional case, the partition function of the GFF (i.e. the ``normalising constant") is given by \cite{Gaw} 
\begin{equation}
Z_{\mathrm{GFF}}(t):=\det(-\Delta_t)^{1/2}=\frac{t}{\pi}|\eta(\frac{it}{\pi})|^2,
\end{equation}
where $\det(-\Delta_t)$ is the z\^eta regularised determinant of the Laplacian (see \cite[Section 1]{OPS} for a general definition and \cite{Gaw} p10 for the value on the torus).

The GFF on $\T_t$ can be constructed using an orthonormal basis of $L^2(\T_t)$ of eigenfunctions of $-\Delta_t$. If $(f_n)_{n\geq0}$ is such a basis with associated eigenvalues $0=\lambda_0<\lambda_1,..,\leq\lambda_n...$, then $(\sqrt{\frac{2\pi}{\lambda_n}}f_n)_{n\geq1}$ is an orthonormal basis of $H^1_0(\T_t)$ and we set
\[X_t:=\sqrt{2\pi}\sum_{n\geq 1}\frac{\alpha_n}{\sqrt{\lambda_n}}f_n,\]
where $(\alpha_n)_{n\geq 1}$ is a sequence of i.i.d. normal random variables. It can be shown that this formal series indeed converges almost surely in $H^{-1}_0(\T_t)$ \cite[Section 3.2]{DRV}. 

As such, the constant coefficient of the GFF (a.k.a. the \emph{zero mode}) depends on the choice of the background metric, since we impose a vanishing mean in the flat metric. In order to get rid of this dependence, we complete the definition of the field by ``sampling" the constant coefficient with Lebesgue measure (see the discussion in \cite[Section 2.2]{DKRV1}). Informally, we can interpret the zero mode as a Gaussian random variable with variance $1/\lambda_0=\infty$ since $\sqrt{\frac{2\pi}{\lambda}}$ times the law of an $\mc{N}(0,\lambda^{-1})$ converges vaguely to Lebesgue as $\lambda\to0$. So we arrive at the field decomposition
\[X=X_t+\frac{c}{\sqrt{t/\pi}}\]
and the final intepretation is that for all continuous functional $F:H^{-1}_0(\T_t)\to\R$, we set\footnote{We add a factor 2 to conform with \cite{KRV2}}
\begin{equation}
\begin{aligned}
\int F(X)e^{-\frac{1}{4\pi}\int_{\T_t}|\nabla X|^2}DX
&=2\:\det(-\Delta_t)^{-1/2}\int_\R\E\left[F(X_t+\frac{c}{\sqrt{t/\pi}})\right]dc\\
&=2(\frac{t}{\pi})^{-1/2}|\eta(\frac{it}{\pi})|^{-2}\int_\R\E[F(X_t+c)]dc.
\end{aligned}
\end{equation}
This formula explains the $t^{-1/2}|\eta(\frac{it}{\pi})|^{-2}$ appearing in the asymptotic formula of Theorem \ref{thm:result}. Applying this to a regularisation of the vertex operator $V_\alpha(0)$ leads to the expression \eqref{eq:def_1_point} of the correlation function $\langle V_\alpha(0)\rangle_{\frac{it}{\pi}}$ \cite[Theorem 4.3]{DRV}.

For the torus, the natural eigenbasis of $L^2(\T_t)$ is given by the functions
\begin{equation}
\begin{aligned}
&f_{n,m}^{ee}(s,\theta):=\frac{1}{\sqrt{(1+\ind_{n=0})(1+\ind_{m=0})\pi t}}\cos(\frac{n\pi s}{t})\cos(m\theta)\\
&f_{n,m}^{eo}(s,\theta):=\frac{1}{\sqrt{(1+\ind_{n=0})\pi t}}\cos(\frac{n\pi s}{t})\sin(m\theta)\\
&f_{n,m}^{oe}(s,\theta):=\frac{1}{\sqrt{(1+\ind_{m=0})\pi t}}\sin(\frac{n\pi s}{t})\cos(m\theta)\\
&f_{n,m}^{oo}(s,\theta):=\frac{1}{\sqrt{\pi t}}\sin(\frac{n\pi s}{t})\sin(m\theta),
\end{aligned}
\end{equation}
and the eigenvalue associated to the eigenfunction $f_{m,n}^{ee,eo,oe,oo}$ is $\lambda_{n,m}:=\frac{n^2\pi ^2}{t^2}+m^2$. Then we can set
\begin{equation}
\label{eq:def_gff}
X_t:=\sqrt{2\pi}\sum_{n,m\neq(0,0)}\frac{\alpha_{n,m}^{ee}}{\sqrt{\lambda_{n,m}}}f_{n,m}^{ee}+\frac{\alpha_{n,m}^{eo}}{\sqrt{\lambda_{n,m}}}f_{n,m}^{eo}+\frac{\alpha_{n,m}^{oe}}{\sqrt{\lambda_{n,m}}}f_{n,m}^{oe}+\frac{\alpha_{n,m}^{oo}}{\sqrt{\lambda_{n,m}}}f_{n,m}^{oo},
\end{equation}
where $\alpha_{n,m}^{ee,eo,oe,oo}$ are i.i.d. centred normal random variables. 

An immediate consequence of this decomposition is that we can sample $X_t$ as follows
\begin{enumerate}
\item Sample a GFF $X_t^\mathrm{D}$ with zero (a.k.a. Dirichlet) boundary conditions\footnote{We refer the reader to \cite[Section 5]{Beb} for an introduction to different types of boundary conditions.} on the cylinder $(0,t)\times\S^1$
\item Sample an independent GFF $X_t^\mathrm{N}$ with free (a.k.a. Neumann) boundary conditions on the cylinder $(0,t)\times\S^1$.
\item For all $(s,\theta)\in(-t,t)\times\S^1$, set $X_t(s,\theta):=\frac{X_t^\mathrm{N}(|s|,\theta)+\mathrm{sign}(s)X_t^\mathrm{D}(|s|,\theta)}{\sqrt{2}}$.
\end{enumerate}
We call this decomposition the \emph{odd/even decomposition of fieds}, which is based on the orthogonal decomposition $H^1_0(\mc{C}_t)=H^{1,e}_0(\mc{C}_t)\oplus H^{1,o}_0(\mc{C}_t)$ where $H^{1,e}_0(\mc{C}_t),H^{1,o}_0(\mc{C}_t)$ are respectively the subspaces of even and odd functions with respect to $s\in(-t,t)$. The nice property of this decomposition is that we can view the GFF $X_t$ on $\T_t$ as a GFF on $\mc{C}_t$ whose odd part is a GFF with zero (Dirichlet) boundary conditions and whose even part is a GFF with free (Neumann) boundary conditions (see \cite[Section 5.1]{Beb} for a discussion of this decomposition).

 Let us now introduce the \emph{radial/angular decomposition of fields} \cite{DMS,KRV2}, which is based on the orthogonal decomposition $H^1_0(\mc{C}_t)=H^{1,R}_0(\mc{C}_t)\oplus H^{1,A}_0(\mc{C}_t)$ (for all $t\in(0,\infty]$) where 
\begin{equation} 
\begin{aligned}
 &H^{1,R}_0(\mc{C}_t)=\{f\in H^1_0(\mc{C}_t),\;f(s,\cdot)\text{ is constant on }\S^1\text{ for all }s\in(-t,t)\}\\
 &H^{1,A}_0(\mc{C}_t)=\{f\in H^1_0(\mc{C}_t),\;\forall s\in(-t,t)\; \int_{\S^1}f(s,\theta)d\theta=0.\}
 \end{aligned}
\end{equation} 
  For a field $X$ on $\mc{C}_\infty$, we will write $X_0(t):=\frac{1}{2\pi}\int_{\S^1}X(t,\theta)d\theta$ for its mean on the circle $\{t\}\times\S^1$ for all $t\in\R$. Viewed in the planar coordinates, $X_0(t)$ is the mean value of $X$ on the circle of radius $e^{-t}$ about 0. 

 Now let $X$ be a GFF on $\mc{C}_{\infty}$, normalised such that $X_0(0)=0$. Then, from \cite[Lemmata 4.2-3]{DKRV2}, we can write $X(t,\theta)=B_t+Y(t,\theta)$ with $B$ independent of $Y$ and
\begin{enumerate}
\item $B_t=X_0(t)$ has the law of a standard two-sided Brownian motion on $\R$.
\item $Y$ is a $\log$-correlated field with covariance kernel
\begin{equation}
\label{eq:covar_y}
H(t,\theta,t',\theta'):=\E[Y(t,\theta)Y(t',\theta')]=\log\frac{e^{-t}\vee e^{-t'}}{|e^{-t+i\theta}-e^{-t'+i\theta'}|}.
\end{equation}
\end{enumerate}

For a GFF $X_t$ on $\T_t$, the radial part is given by the sum of the radial parts of $X_t^\mathrm{D}$ and $X_t^\mathrm{N}$. Hence $(\sqrt{2}B_s)_{0\leq s\leq t}$ is the independent sum of a Brownian bridge and a standard Brownian motion with its mean subtracted.

	\subsection{Gaussian Multiplicative Chaos}
	\label{subsubsec:gmc}
Recall that a GFF $X$ (on $\mc{C}_\infty$ or $\T_t$) is only defined as a distribution, so the exponential term $e^{\gamma X}$ is ill-defined \textit{a priori}. However it is possible to make sense of the measure $e^{\gamma X(s,\theta)}dsd\theta$ using a regularising procedure based on Kahane's theory of Gaussian Multiplicative Chaos (GMC) (see \cite{RVb,Be} for more detailed reviews). We use the regularisation called the circle average: let $X_\varepsilon(x)$ be a jointly continuous version of the average of the field on the circle of (Euclidean) radius $\varepsilon$ about $x\in M$ \cite[Section 2]{Beb}. From \cite[Proposition 2.6]{DKRV1} and \cite[Proposition 3.8]{DRV}, the sequence of measures 
\begin{equation}\label{eq:regularised_chaos}
dM^\gamma_\varepsilon(x):=e^{\gamma X_\varepsilon(x)-\frac{1}{2}\gamma^2\E[X_\varepsilon(x)^2]}dx
\end{equation}
converges in probability as $\varepsilon\to0$ (in the sense of weak convergence of measures) to an almost surely non-trivial measure $dM^\gamma$ with no atoms, for all $\gamma\in(0,2)$. Moreover, the result of \cite[Theorem 1.1]{Be} together with universality of the limit (see the discussion in \cite{Beb} after Theorem 2.1) shows that $M^\gamma_\varepsilon(D)\to M^\gamma(D)$ in $L^1$ as $\varepsilon\to0$ for all Borel set $D$.
%


An important tool in GMC is  Kahane's convexity inequality, which we will use in form of Theorem \ref{thm:convexity} below. In this form, this theorem is a consequence of \cite[Theorem 2.1]{RVb} (see the discussion after Theorem 2.3 of \cite{RVb}).
\begin{theorem}\cite[Theorem 2.1]{RVb}\label{thm:convexity}
Let $X$ and $Y$ be two continuous Gaussian fields on $D\subset\C$ such that for all $x,y\in D$
\[\E[X(x)X(y)]\leq\E[Y(x)Y(y)].\]
Then for all convex function $F:\R_+\to\R$ with at most polynomial growth at infinity,
\[\E\left[F\left(\int_De^{\gamma X(x)-\frac{\gamma^2}{2}\E[X(x)^2]}dx\right)\right]\leq\E\left[F\left(\int_De^{\gamma Y(x)-\frac{\gamma^2}{2}\E[Y(x)^2]}dx\right)\right].\]
\end{theorem}
In practice, we can apply this result to the GMC measures of $\log$-correlated fields (like the GFF) using the regularising procedure. Suppose $X,Y$ are $\log$-correlated fields with $|\E[X(x)X(y)-\E[Y(x)Y(y)]|\leq\varepsilon$ for all $x,y$, and write $M^\gamma,N^\gamma$ for their respective chaos measure. In particular we have
\[\E[X(x)X(y)]\leq \E[Y(x)Y(y)]+\varepsilon.\]
Notice that the field $Z(x)=Y(x)+\sqrt{\varepsilon}\delta$ -- with $\delta\sim\mc{N}(0,1)$ independent of everything -- has covariance kernel $\E[Y(x)Y(y)]+\varepsilon$. The argument of \cite{RVb} in the discussion following Theorem 2.3 shows that we can apply Kahane's inequality in the limit, and we get:
\[\E[M^\gamma(D)^{-r}]\leq \E[e^{-r\gamma\sqrt{\varepsilon}\delta}N^\gamma(D)^{-r}]=e^{\frac{1}{2}\gamma^2r^2\varepsilon}\E[N^\gamma(D)^{-r}].\]
By symmetry of the roles of $X$ and $Y$, the converse inequality is also true, so that in the end
\[\E[M^\gamma(D)^{-r}]=\E[N^\gamma(D)^{-r}](1+O_{\varepsilon\to0}(\varepsilon)).\]

	\subsection{Derivation of the correlation function}
	\label{subsec:derivation}
Using the GFF and GMC, we give a short derivation of the correlation function $\langle V_\alpha(0)\rangle_\frac{it}{\pi}$ on the torus. In \cite{DRV}, this object is constructed so as to satisfy some invariance properties, e.g. the Weyl anomaly (Theorem 4.1) and modular invariance (Theorem 4.6). Hence, as in \cite[Section 2.2]{KRV2}, we suppose that we have fixed the geometric setting described above (Green's function $G_t$, representative of the modulus $\tau=\frac{it}{\pi}$) and take the invariance properties as part of the definition.

 We renormalise the vertex operator $V_\alpha(0)$ by setting 
\begin{equation}\label{eq:regularised_vertex}
V_{\alpha,\varepsilon}(0):=e^{\alpha X_\varepsilon(0)-\frac{\alpha^2}{2}\E[X_\varepsilon(0)^2]}.
\end{equation} 
Applying Girsanov's theorem, then taking $\varepsilon\to0$ and making the change of variables $u=e^{\gamma c}$ we can set:
 
\begin{equation}\label{eq:def_correl}
\begin{aligned}
\langle V_\alpha(0)\rangle_{\frac{it}{\pi}}
&:=\underset{\varepsilon\to0}{\lim}\,2(\frac{t}{\pi})^{-1/2}|\eta(\frac{it}{\pi})|^{-2}\int_\R e^{\alpha c}\E\left[e^{\alpha X_\varepsilon(0)-\frac{\alpha^2}{2}\E[X_\varepsilon(0)^2]}\exp(-\mu e^{\gamma c}M^\gamma(\T_t))\right]dc\\
&=2(\frac{t}{\pi})^{-1/2}|\eta(\frac{it}{\pi})|^{-2}\int_\R e^{\alpha c}\E\left[\exp(-\mu e^{\gamma c}\int_{\T_t}e^{\gamma\alpha G_t(0,\cdot)}dM^\gamma)\right]dc\\
&=2(\frac{t}{\pi})^{-1/2}|\eta(\frac{it}{\pi})|^{-2}\gamma^{-1}\mu^{-\frac{\alpha}{\gamma}}\Gamma(\frac{\alpha}{\gamma})\E\left[\left(\int_{\T_t}e^{\gamma\alpha G_t(0,\cdot)}dM^\gamma\right)^{-\frac{\alpha}{\gamma}}\right].
\end{aligned}
\end{equation}

\begin{remark}
At first glance, our choices of renormalisation in \eqref{eq:regularised_chaos} and \eqref{eq:regularised_vertex} may seem different than the ones in \cite{DRV} where the renormalisation factors are $\varepsilon^\frac{\gamma^2}{2}$ and $\varepsilon^\frac{\alpha^2}{2}$ respectively. However, notice that for the lateral noise $Y$ on the infinite cylinder $\mc{C}_\infty$, we have $\E[Y^2_\varepsilon(x)]=\log\frac{1}{\varepsilon}+o(1)$ with $o(1)$ uniform, so we get the same limit by Kahane's inequality. Moreover, our parametrisation is made precisely to have the Green function for the lateral noise $Y_t$ on $\T_t$ converging in a suitable sense to that of $Y$ as $t\to\infty$ (see Section \ref{subsec:spheric_equiv_toric} and in particular \eqref{eq:green_comparison}).
\end{remark}

%

\section{Proofs}
\label{sec:proof}

	\subsection{Proof of Proposition \ref{prop:spheric_limit}}
		\label{subsec:spheric_limit}
		We will start by showing in Section \ref{subsubsec:background} that $t\langle V_Q(0)V_\alpha(1)V_Q(\infty)\rangle_t$ has a limit $t\to\infty$ and find its expression in terms of the derivative DOZZ formula in Section \ref{subsubsec:characterisation}. Section \ref{subsubsec:heuristic} gives a heuristic explanation for this limit.

		\subsubsection{Background and notations}\label{subsubsec:background}
Let $g(z)=|z|_+^{-4}$ be the cr\^epe metric on $\hat{\C}$. Under the conformal change of coordinates $\psi:\mc{C}_\infty\to\hat{\C}$ defined by $\psi(z)=-\log z$, we get the metric $g_{\psi}(t,\theta)=e^{-2|t|}$ on the infinite cylinder.

Let $X(t,\theta)=B_t+Y(t,\theta)$ be a GFF on $\mc{C}_\infty$. By conformal covariance \cite[Equation (3.13)]{GRV}, taking the chaos of $X$ with respect to $g_\psi$ is the same as taking the chaos of $X(t,\theta)-Q|t|$ with respect to Lebesgue measure. Equivalently, this is the same as taking the radial part of the GFF to be the drifted Brownian motion $B_t-Q|t|$. Notice that the angular part is unchanged in this process and we write $dN^\gamma$ for the GMC measure of $Y$. We will be interested in the negative moments of GMC. To this end, for all $t<t'$, we introduce the random variable
\begin{equation}\label{eq:def_Z}
Z_{t,t'}(\lambda):=\int_{[t,t']\times\S^1}e^{\gamma(B_s+(\lambda-Q)|s|+\alpha G(0,s+i\theta))}dN^\gamma(s,\theta),
\end{equation}
where $r>0$ is fixed throughout the proof and recall $G(\cdot,\cdot)$ is Green's function on $\mc{C}_\infty$. For notational convenience, we also define $Z_t(\lambda):=Z_{-t,t}(\lambda)$.

We can see in the expression of $Z_t(\lambda)$ that the Brownian motion has a drift that makes the chaos measure integrable when $|t|\to\infty$. The value of the drift is precisely linked to the strength of the singularity and in vanishes when $\lambda=Q$, causing the mass to explode and the negative moments to vanish, so we have $Z_\infty(Q):=\underset{t\to\infty}{\lim}Z_t(Q)=\infty$ a.s., and $\underset{t\to\infty}{\lim}\E[Z_t(Q)^{-r}]=0$ \cite[Theorem 3.2]{DKRV1}. On the other hand, $Z_t(\lambda)$ converges a.s. to a positive, finite random variable $Z_\infty(\lambda)$ for all $\lambda\in(Q-\frac{\alpha}{2},Q)$, and all negative moments of $Z_\infty(\lambda)$ are positive and finite. Furthermore, the DOZZ formula states that for all $\lambda\in(Q-\frac{\alpha}{2},Q)$, we have 
\[C_\gamma(\lambda,\alpha,\lambda)=2\gamma^{-1}\mu^{-\frac{\alpha}{\gamma}}\Gamma\left(\frac{\alpha}{\gamma}\right)\E\left[Z_\infty(\lambda)^{-\frac{\alpha}{\gamma}}\right].\]

The rate at which the negative moments of $Z_t(Q)$ vanish with $t$ was studied in \cite{DKRV2} where it was shown that $t\E[Z_t(Q)^{-r}]$ has a non-trivial limit as $t\to\infty$ (Theorem 2.1 with $k=2$ and $t=\log\frac{1}{\varepsilon}$). Let us briefly recall what the strategy was, as we will need some ingredients from the proof. What we state from here to equation \eqref{eq:def_Q_puncture} is the idea of the proof of Proposition 3.1 of \cite{DKRV2}. For $b,t>0$, define the event 
\begin{equation}
\label{eq:event}
A_{b,t}:=\left\lbrace\underset{-t\leq s\leq t}{\sup}B_s<b\right\rbrace.
\end{equation}
By independence of the Brownian motions $(B_t)_{t\geq0}$ and $(B_{-t})_{t\geq0}$, we have
\begin{equation}
\P(A_{b,t})=\left(2\int_0^{b/\sqrt{t}}\frac{e^{-\frac{x^2}{2}}}{\sqrt{2\pi}}dx\right)^2=:f(b/\sqrt{t})^2,
\end{equation}
with the elementary estimates $f(x)\to 1$ as $x\to\infty$ and
\begin{equation}\label{eq:estimate_P}
f(x)\sim\sqrt{\frac{2}{\pi}}x\text{ as }x\to0.
\end{equation}

The law of $(b-B_s)_{-t\leq s\leq t}$ converges as $t\to\infty$ to a two-sided, 3-dimensional Bessel process on $\R$ taking the value $b$ at $t=0$ \cite[Lemma 4.5]{DKRV2} and the independence of the left and right processes). Hence the limiting process $(B_s)_{s\in\R}$ goes to $-\infty$ as $|s|\to\infty$ at scale roughly $-\sqrt{|s|}$.

Let $\P_b$ be the law of the GFF on $\mc{C}_\infty$ where the radial part is replaced by $b$ minus a 2-sided, 3-dimensional Bessel process taking the value $b$ at $t=0$. Under $\P_b$, $Z_\infty(Q)$ is a.s. a non-trivial random variable, and $\E_b[Z_\infty(Q)]<\infty$ \cite[Equations (5.5) and (5.6)]{DKRV2}. Furthermore the authors show that $\E_b\left[Z_\infty(Q)^{-r}\right]\in(0,\infty)$ and its value is characterised by \cite[Proposition 3.1]{DKRV2}:
\begin{equation}\label{eq:b_t}
\begin{aligned}
\underset{t\to\infty}{\lim}t\E\left[Z_t(Q)^{-r}\right]
&=\underset{t\to\infty}{\lim}\underset{b\to\infty}{\lim}t\E[Z_t(Q)^{-r}\ind_{A_{b,t}}]\\
&=\underset{t\to\infty}{\lim}\underset{b\to\infty}{\lim}tf(b/\sqrt{t})^2\E\left[\left.Z_t(Q)^{-r}\right|A_{b,t}\right]\\
&=\underset{b\to\infty}{\lim}\underset{t\to\infty}{\lim}tf(b/\sqrt{t})^2\E\left[\left.Z_t(Q)^{-r}\right|A_{b,t}\right]\\
&=\frac{2}{\pi}\underset{b\to\infty}{\lim}b^2\E_b\left[Z_\infty(Q)^{-r}\right].
\end{aligned}
\end{equation}
The exchange of limits in $b$ and in $t$ is justified by the uniform convergence in $b$ with respect to $t$. In the last line, the limit in $b$ can be shown to be finite using estimates on hitting probabilities of Bessel processes.

 		\subsubsection{Characterisation of the limit}\label{subsubsec:characterisation}
We now turn to the study of the behaviour of $\E[Z_\infty(\lambda)^{-r}]$ as $\lambda\to Q$. From the independence of the left and right radial processes, it suffices to study the one-sided problem and show that $\frac{\E[Z_{0,\infty}(\lambda)^{-r}]}{2(Q-\lambda)}$ has a limit as $\lambda\to Q$ and that this limit coincides with $\underset{b\to\infty}{\lim}b\E_b[Z_{0,\infty}(Q)^{-r}]$.

Let $\lambda\in(Q-\frac{\alpha}{2},Q)$. By the Williams path decomposition (see e.g. \cite[Lemma 2.6]{KRV2}), we can sample a Brownian motion in $\R_+$ with drift $\lambda-Q<0$ as follows:
\begin{enumerate}
\item Sample an exponential random variable $M\sim\mathrm{Exp}(2(Q-\lambda))$ (this is the supremum of the process).
\item Conditionally on $M$, run an independent Brownian motion with drift $Q-\lambda>0$ until its hitting time $T_{\lambda,b}$ of $b$.
\item Conditionally on $T_{\lambda,b}$, run an independent Brownian motion in $[T_{\lambda,b},\infty)$ with drift $\lambda-Q<0$ started from $b$ and conditioned to stay below $b$.
\end{enumerate}
By definition, what is meant by Brownian motion with drift $\nu>0$ conditioned to stay positive is the process with generator $\frac{1}{2}\frac{d^2}{dx^2}+\nu\cot(\nu x)\frac{d}{dx}$ \cite[Section 12.4]{KRV2}. In the limit $\nu\to0$, we get the generator $\frac{1}{2}\frac{d^2}{dx^2}+\frac{1}{x}\frac{d}{dx}$ of the $3d$ Bessel process. Thus, on the event that $M=b$, the Williams path decomposition converges in law as $\lambda\to Q$ to the joining of a Brownian motion run until its hitting time of $b$ and a Brownian motion conditioned to stay below $b$ (i.e. $b$ minus a $3d$ Bessel process). Thus, Williams' path decomposition gives a way to make sense of conditioning on the value of the supremum of the radial process, and we can write for all $r>0$,
\[\frac{\E\left[Z_{0,\infty}(\lambda)^{-r}\right]}{2(Q-\lambda)}=\int_0^\infty\E\left[Z_{0,\infty}(\lambda)^{-r}|M=b\right]e^{2b(\lambda-Q)}db.\] 
As already seen in Section \ref{subsubsec:background}, the properties of the Bessel process imply that, for all $b>0$, $\E[Z_{0,\infty}(Q)^{-r}|M=b]:=\underset{\lambda\to Q}{\lim}\E[Z_{0,\infty}(\lambda)^{-r}|M=b]$ exists and is positive. Furthermore, the positivity of the GMC measure implies 
\begin{equation}\label{eq:exp_decay}
\E\left[Z_{0,\infty}(\lambda)^{-r}|M=b\right]\leq\E\left[Z_{\tau_{b-1},\tau_b}(\lambda)^{-r}|M=b\right]\leq e^{-r\gamma(b-1)}\E\left[Z_{0,\tau_1}(\lambda)^{-r}|M=1\right],
\end{equation}
where we wrote $\tau_x$ for the hitting time of $x$ by the drifted process, and used the Markov property and the stationarity of the lateral noise. From \cite[Lemma 4.4]{DKRV2}, we know that $\E[Z_{0,\tau_1}(Q)^{-r}|M=b]<\infty$. Actually, this lemma also holds in the case $\lambda<Q$ since it relies on an estimate of $\P(\tau_1<t)$ as $t\to0$ which gives the same result in the drifted case. This implies that $\E\left[Z_{0,\infty}(\lambda)^{-r}|M=b\right]$ decays exponentially fast as $b\to\infty$. By stochastic domination \cite[Section 9.2]{KRV2}, $\E[Z_{0,\infty}(\lambda)^{-r}|M=b]$ is also decreasing in $\lambda$ for all $b$. It then follows from the dominated convergence theorem that
\begin{equation} \label{eq:lim_lambda}
 \underset{\lambda\to Q}{\lim}\frac{\E\left[Z_{0,\infty}(\lambda)^{-r}\right]}{2(Q-\lambda)}=\int_0^\infty\E\left[Z_{0,\infty}(Q)^{-r}|M=b\right]db.
 \end{equation}
 
 To conclude, we must show that this limit coincides with $\underset{b\to\infty}{\lim}b\E_b[Z_\infty(Q)^{-r}]$. Under $\P_b$, $(b-B_s)_{s\geq0}$ is a $3d$-Bessel process started from $b$, so $(b-B_s)^{-1}$ is a positive continuous local martingale a.s. converging to 0 as $s\to\infty$. Applying the optional stopping theorem, we find that $\P_b(\sigma_x<\infty)=\frac{x}{b}$ for all $x\in(0,b)$, where $\sigma_x$ is the first hitting time of $x$ by $(b-B_s)_{s\geq0}$ (this is the well-known fact that the infimum of a Bessel process started from $b>0$ is uniformly distributed in $(0,b)$, see also \cite[Exercise 2.5]{MY} for a more general setting). It follows that under $\P_b$, $M=\underset{s\geq0}{\sup}B_s$ is uniformly distributed in $[0,b]$. Hence
\begin{equation}\label{eq:equal_limits}
 \underset{b\to\infty}{\lim}b\E_b\left[Z_{0,\infty}(Q)^{-r}\right]=\underset{b\to\infty}{\lim}\int_0^b\E_b\left[Z_{0,\infty}(Q)^{-r}|M=b'\right]db'=\int_0^\infty\E\left[Z_{0,\infty}(Q)^{-r}|M=b\right]db.
 \end{equation}

Thus we find the same limit as in \eqref{eq:lim_lambda}. Now we go back to the two-sided setting. Since the left and right radial processes are i.i.d. Brownian motions, we can apply the above conditioning to each one of these processes independently, and putting together \eqref{eq:b_t}, \eqref{eq:lim_lambda} and \eqref{eq:equal_limits} yields: 
\[\underset{\lambda\to Q}{\lim}\frac{\E\left[Z_\infty(\lambda)^{-r}\right]}{4(Q-\lambda)^2}=\underset{t\to\infty}{\lim}\frac{\pi t}{2}\E\left[Z_t(Q)^{-r}\right].\]

 Plugging this into the expression for the correlation function yields
\begin{equation}
\begin{aligned}
\frac{\pi}{2}\underset{t\to\infty}{\lim}t\langle V_Q(0)V_\alpha(1)V_Q(\infty)\rangle_t
&=2\gamma^{-1}\mu^{-\frac{\alpha}{\gamma}}\Gamma\left(\frac{\alpha}{\gamma}\right)\underset{t\to\infty}{\lim}\frac{\pi}{2}t\E\left[Z_t^{-r}\right]\\
&=2\gamma^{-1}\mu^{-\frac{\alpha}{\gamma}}\Gamma\left(\frac{\alpha}{\gamma}\right)\underset{\lambda\to Q}{\lim}\frac{\E\left[Z_\infty(\lambda)^{-r}\right]}{4(\lambda-Q)^2}\\
&=\frac{1}{4}\underset{\lambda\to Q}{\lim}\frac{C_\gamma(\lambda,\alpha,\lambda)}{(\lambda-Q)^2}.
\end{aligned}
\end{equation}\qed

		\subsubsection{Heuristic interpretation of the limit}\label{subsubsec:heuristic}
For the record, we give a heuristic interpretation of the result of Proposition \ref{prop:spheric_limit}. Using the expression of the Radon-Nikodym derivative of the Bessel process with respect to Brownian motion, one can rewrite \eqref{eq:b_t} as
\begin{equation}
\label{eq:def_Q_puncture}
\underset{t\to\infty}{\lim}t\E\left[Z_t^{-r}\right]=\frac{2}{\pi}\underset{t\to\infty}{\lim}\E\left[B_tB_{-t}Z_t^{-r}\right]
\end{equation}
and we define the (renormalised) correlation function to be
\begin{equation}
\label{eq:renormalised_correl}
^R\langle V_Q(0)V_\alpha(1)V_Q(\infty)\rangle_{\S^2}:=2\gamma^{-1}\mu^{-\frac{\alpha}{\gamma}}\Gamma\left(\frac{\alpha}{\gamma}\right)\underset{t\to\infty}{\lim}\E\left[B_tB_{-t}Z_t^{-r}\right].
\end{equation}

 We have seen that this correlation function can be expressed using the derivative of DOZZ formula at the critical point $\alpha_1=\alpha_3=Q$. The usual interpretation of $^R\langle V_Q(0)V_\alpha(1)V_Q(\infty)\rangle_{\S^2}$ is that of a derivative operator. Indeed, the value of $B_t$ in equation \eqref{eq:renormalised_correl} is the average of the field on the circle of radius $e^{-t}$ about $0$, so it is formally $X(0)$ in the limit $t\to\infty$. Still on the formal level, we have the interpretation
\begin{equation}
\begin{aligned}
^R\langle V_Q(0)V_\alpha(1)V_Q(\infty)\rangle
&=\langle X(0)V_Q(0)V_\alpha(1)X(\infty)V_Q(\infty)\rangle_{\S^2}\\
&=\left\langle\frac{d}{d\lambda}V_\lambda(0)_{|\lambda=Q}V_\alpha(1)\frac{d}{d\lambda}V_\lambda(\infty)_{|\lambda=Q}\right\rangle_{\S^2}.
\end{aligned}
\end{equation}

This explains why we could expect $^R\langle V_Q(0)V_\alpha(1)V_Q(\infty)\rangle_{\S^2}$ to be expressed in terms of the (second) derivative of DOZZ formula at the critical point.

\subsection{Proof of Proposition \ref{prop:spheric_equiv_toric}}
\label{subsec:spheric_equiv_toric}
The second item in the proof of Theorem \ref{thm:result} is the equivalent asymptotic behaviour of $\langle V_\alpha(0)\rangle_\frac{it}{\pi}$ and $\langle V_Q(0)V_\alpha(1)V_Q(\infty)\rangle_t$. This will follow from comparisons between Green's function on the infinite cylinder and the torus.

\begin{lemma}
\label{lem:torus_decomposition}
Let $X_t$ be a GFF on the torus $\T_t$ (embedded into $\mc{C}_t$) with the normalisation $\int_{\S^1}X_t(0,\theta)d\theta=0$. Then we can write $X_t(s,\theta)=B_t(s)+Y_t(s,\theta)$ with $B_t$ independent of $Y_t$ and
\begin{enumerate}
\item For all $s\in(-t,t)$, $B_t(s)=\frac{B^e(|s|)+\mathrm{sign}(s)B^o(|s|)}{\sqrt{2}}$ where $(B^o(s))_{0\leq s\leq t}$ is standard Brownian bridge and $(B^e(s))_{0\leq s\leq t}$ is an independent standard Brownian motion.
\item $Y_t$ is a $\log$-correlated Gaussian field with covariance kernel (recall equation \eqref{eq:covar_y})
\begin{equation}
\label{eq:covar_periodic}
H_t(s,\theta,s',\theta')=\sum_{n\in\Z}H(s,\theta,s'+2nt,\theta').
\end{equation}
\end{enumerate}
\end{lemma}

\begin{proof}
With the choice of normalisation of the Lemma, we can sample $X_t$ simply by setting $X_t:=\tilde{X}_t-\int_{\S^1}X_t(0,\theta)d\theta$ where $\tilde{X}_t$ is a GFF on $\T_t$ with vanishing mean on $\T_t$. From Section \ref{subsubsec:gff}, the radial part of $\tilde{X}_t$ on $(0,t)\times\S^1$ is $\frac{B^o(s)+B^e(s)}{\sqrt{2}}$ where $(B^o(s))_{0\leq s\leq t}$ is a standard Brownian bridge and $B^e(s)$ is an independent Brownian motion whose mean has been subtracted. The normalisation of $X_t$ is simply translating $B^o$ along the $y$ axis such that $B_t^o(0)=0$, so the radial part is the claimed one.

Now we deal with the angular part $X_t$. From equation \eqref{eq:covar_y}, we have for all $s\in(-t,t)$, $\theta\in\S^1$ and $n\in\Z\setminus\{0\}$
\[H(0,0,s+2nt,\theta)=\log\frac{1}{|1-e^{-|s+2n|t-i\theta}|}=O_{|n|\to\infty}(e^{-2|n|t}),\]
implying that the series \eqref{eq:covar_periodic} converges absolutely on compact subsets of $\mc{C}_t\setminus\{(s,\theta)\}$ for all $(s,\theta)\in\mc{C}_t$ (we used the translation invariance of $H$). In particular, $H_t(s,\theta,\cdot,\cdot)$ defines a function on $\T_t$.

For all $(s,\theta)\in\mc{C}_t$, the function $(s',\theta')\mapsto\sum_{n\neq0}H(s,\theta,s'+2nt,\theta')$ defined on $\mc{C}_t$ is an absolutely convergent sum of harmonic functions on $\mc{C}_t$ (with respect to the Laplacian on $\mc{C}_\infty$), and the second derivatives also converge absolutely. Hence the function is harmonic on $\mc{C}_t$. Note also that $H_t(s,\theta,\cdot,\cdot)$ is a sum of angular functions, so it is also angular. Let $\varphi\in\mc{C}^\infty(\T_t)$ be an angular function. We can view $\varphi$ as a $2t$-periodic function on $\mc{C}_\infty$ and we have $\langle-\frac{1}{2\pi}\Delta_tH_t(s,\theta,\cdot,\cdot),\varphi\rangle=\langle\frac{-1}{2\pi}\Delta H(s,\theta,\cdot,\cdot),\varphi\rangle=\varphi(s,\theta)$. So by definition $H_t$ is the angular part of Green's function on $\T_t$.
\end{proof}
\begin{proof}[Proof of Proposition \ref{prop:spheric_equiv_toric}]
Let us introduce some notation. Fix $\delta>0$ and write
\begin{equation}
Z_t:=\int_{\mc{C}_t}e^{\gamma(B(s)+\alpha G(0,s+i\theta))}dN^\gamma(s,\theta)=U_t+\xi_t,
\end{equation}
where
\begin{equation}
\begin{aligned}
&U_t=\int_{\mc{C}_{t^{1-\delta}}}e^{\gamma(B(s)+\alpha G(0,s+i\theta))}dN^\gamma(s,\theta)\\
&\xi_t=\int_{(-t,-t^{1-\delta})\cup(t^{1-\delta},t)\times\S^1}e^{\gamma(B(s)+\alpha G(0,s+i\theta))}dN^\gamma(s,\theta)
\end{aligned}
\end{equation}
We define also
\[\tilde{Z}_t:=\int_{\T_t}e^{\gamma(B_t(s)+\alpha G_t(0,s+i\theta))}dN_t^\gamma(s,\theta)=\tilde{U}_t+\tilde{\xi}_t\]
where $\tilde{U}_t$ and $\tilde{\xi}_t$ are defined similarly (here $dN^\gamma_t$ is the GMC measure of the field $Y_t$). The term $\tilde{U}_t$ is the core of the mass while $\tilde{\xi}_t$ is some error term that we have to control. We will see that $\tilde{U}_t$ behaves exaclty as $U_t$ as $t\to\infty$.

It follows from Lemma \ref{lem:torus_decomposition} that for all $x,y\in\mc{C}_{t^{1-\delta}}$
\begin{equation}
\label{eq:green_comparison}
|H_t(x,y)-H(x,y)|=\left|\sum_{n\neq0}H(x,y+2nt)\right|\leq Ce^{-2t}
\end{equation}
for some constant $C>0$ independent of $t$.

Let $b>0$ and define the event
\[\tilde{A}_{b,t}:=\left\lbrace \underset{-t\leq s\leq t}{\sup}B_t(s)<b\right\rbrace.\]
By Brownian scaling, there exists a function $g:\R_+\to[0,1]$ such that $\P\left(\tilde{A}_{b,t}\right)=g(b/\sqrt{t})$. It is clear that $\underset{x\to\infty}{\lim}g(x)=1$ and we will show in Lemma \ref{lem:conditioning} (at the end of this section) that $g(x)\underset{x\to0}{\sim}\frac{3}{\pi}x^2$.

On $\tilde{A}_{b,t}$, the process $(b-B_t(s))_{0\leq s\leq t}$ is absolutely continuous with respect to a 3d-Bessel process started from $b$. Hence there exists $\nu>0$ such that the event 
\[\left\lbrace \forall s\in (t^{1-\delta},t)\cup(-t,-t^{1-\delta}),\;B_t(s)\leq -t^{1/2-\nu}\right\rbrace\]
 occurs with high probability as $t\to\infty$, implying that $\tilde{\xi}_t\to0$ in probability conditionally on $\tilde{A}_{b,t}$ as $t\to\infty$.
Similarly, $\xi_t\to0$ in probability as $t\to\infty$ when conditioned on $A_{b,t}$.

From the previous subsection we know that $Z_t$ conditioned on $A_{b,t}$ has a non-trivial limit $Z_\infty$ as $t\to\infty$, and the negative moments of $Z_\infty$ are finite. Now for each $\varepsilon>0$, we have
\begin{equation}
\label{eq:xi_out}
\E[U_t^{-r}|A_{b,t}]\geq\E[Z_t^{-r}|A_{b,t}]\geq\E[(U_t+\varepsilon)^{-r}\ind_{\xi_t<\varepsilon}|A_{b,t}],
\end{equation}
and taking first $t\to\infty$ then $\varepsilon\to0$ yields
\[\underset{t\to\infty}{\lim}\E[U_t^{-r}|A_{b,t}]=\underset{t\to\infty}{\lim}\E[Z_t^{-r}|A_{b,t}].\]

We now turn to the study of $\tilde{U}_t$. Let $\mc{E}_t$ be the Radon-Nikodym derivative of the law of the process $(B_t(s))_{-t^{1-\delta}\leq s\leq t^{1-\delta}}$ (conditioned on $\tilde{A}_{b,t}$) with respect to that of the process $(B(s))_{-t^{1-\delta}\leq s\leq t^{1-\delta}}$ (conditioned on $A_{b,t}$). From Lemma \ref{lem:torus_decomposition}, this is the Radon-Nikodym derivative of the Brownian bridge $B^{o}$ in $[0,t]$ stopped at $t^{1-\delta}$ with respect to Brownian motion in $[0,t^{1-\delta}]$. From \cite[Exercise 9.4]{MY}, this is explicitely given by $(1-t^{-\delta})^{-1/2}e^{-\frac{(B^o_{t^{1-\delta}})^2}{2(t-t^{1-\delta})}}$, so $\mc{E}_t\to1$ a.s. and in $L^1$. Thus:
\begin{equation}
\label{eq:lim_core}
\begin{aligned}
\E\left[\tilde{U}_t^{-r}|\tilde{A}_{b,t}\right]
&=\E\left[\mc{E}_t\left.\left(\int_{\mc{C}_{t^{1-\delta}}}e^{\gamma(B(s)+\alpha G_t(0,s+i\theta))}dN^\gamma_t(s,\theta)\right)^{-r}\right|A_{b,t}\right]\\
&=\E\left[\mc{E}_tU_t^{-r}|A_{b,t}\right](1+O(e^{-2t})),
\end{aligned}
\end{equation}
where we have used the estimate \eqref{eq:green_comparison} and Kahane's convexity inequality (Section \ref{subsubsec:gmc}) to go from $Y_t$ (resp. $G_t(0,\cdot)$) to $Y$ (resp. $G(0,\cdot)$). Hence
\[\underset{t\to\infty}{\lim}\E[\tilde{U}_t^{-r}|\tilde{A}_{b,t}]=\underset{t\to\infty}{\lim}\E[U_t^{-r}|A_{b,t}].\]

Since $\tilde{\xi}_t\to0$ in probability conditionally on $\tilde{A}_{b,t}$, we find using the same argument as in \eqref{eq:xi_out}
\begin{equation}\label{eq:tilde_no_tilde}
\underset{t\to\infty}{\lim}\E[\tilde{Z}_t^{-r}|\tilde{A}_{b,t}]=\underset{t\to\infty}{\lim}\E[\tilde{U}_t^{-r}|\tilde{A}_{b,t}]=\underset{t\to\infty}{\lim}\E[U_t^{-r}|A_{b,t}]=\underset{t\to\infty}{\lim}\E[Z_t^{-r}|A_{b,t}].
\end{equation}

Finally, we want to take the limit $b\to\infty$ in the above equation and then exchange the order of the limits. This is the argument of \cite{DKRV2} leading to \eqref{eq:b_t} but we briefly recall it for completeness. Recall that $\E[Z_t^{-r}|\sup_{0\leq s\leq t}B_s=x]=e^{-\gamma xr}\times O(1)$ as $x\to\infty$, where $O(1)$ is independent of $t>0$. This is because factorising out the maximum gives a contribution $e^{-bxr}$ on this event (see also \eqref{eq:exp_decay}). Moreover, the law of $\sup_{0\leq s\leq t}B_s$ conditionally on $\{\sup_{0\leq s\leq t}B_s<b\}$ is absolutely continuous with respect to $\frac{\ind_{(0,b)}dx}{b}$ (the uniform measure on $(0,b)$), and the density is uniformly bounded in $t>0$. Thus, the convergence of $b^2\E[Z_t^{-r}|A_{b,t}]$ as $b\to\infty$ is exponentially fast, with a rate $O(e^{-\gamma br})$ independent of $t>0$. This uniform convergence enables to exchange limits, and with the estimate \eqref{eq:estimate_P} we find $\underset{b\to\infty}{\lim}\underset{t\to\infty}{\lim}\,b^2\E[Z_t^{-r}|A_{b,t}]=\frac{\pi}{2}\underset{t\to\infty}{\lim}\,t\E[Z_t^{-r}]$. The same argument applies to $\tilde{Z}_t$, and Lemma \ref{lem:conditioning} then entails:


\begin{equation}
\label{eq:lim_bridge}
\begin{aligned}
\underset{t\to\infty}{\lim}\frac{\pi}{3}t\E[\tilde{Z}_t^{-r}]
=\underset{b\to\infty}{\lim}b^2\underset{t\to\infty}{\lim}\E\left[\tilde{Z}_t^{-r}|\tilde{A}_{b,t}\right]
&=\underset{b\to\infty}{\lim}b^2\underset{t\to\infty}{\lim}\E\left[Z_t^{-r}|A_{b,t}\right]\\
&=\underset{t\to\infty}{\lim}\frac{\pi}{2}t\E\left[Z_t^{-r}\right],
\end{aligned}
\end{equation}
i.e. $\underset{t\to\infty}{\lim}t\E\left[\tilde{Z}_t^{-r}\right]=\frac{3}{2}\underset{t\to\infty}{\lim}t\E\left[Z_t^{-r}\right]$.

\end{proof}

We conclude this section by stating and proving Lemma \ref{lem:conditioning}.
\begin{lemma}
\label{lem:conditioning}
Let $(B_t)_{-1\leq t\leq 1}$ be a standard 2-sided Brownian motion. Then
\[\P\left(\left.\underset{-1\leq t\leq 1}{\sup}B_t<x\right|B_1=B_{-1}\right)\underset{x\to0}{\sim}\frac{3}{\pi}x^2,\]
where we abuse notation by writing $\P(\;\cdot\;|B_1=B_{-1})=\underset{\varepsilon\to0}{\lim}\P(\;\cdot\;|\;|B_1-B_{-1}|\leq\varepsilon)$.
\end{lemma}

\begin{proof}
For $\varepsilon>0$ we have
\begin{equation}
\label{eq:estimate}
\begin{aligned}
\P\left(\left.\underset{-1\leq t\leq 1}{\sup}B_t<x\right||B_1-B_{-1}|<\varepsilon\right)
&=\P\left(\underset{-1\leq t\leq 1}{\sup}B_t<x\right)\frac{\P\left(|B_1-B_{-1}|<\varepsilon\left|\underset{-1\leq t\leq 1}{\sup}B_t<x\right.\right)}{\P\left(|B_1-B_{-1}|<\varepsilon\right)}.
\end{aligned}
\end{equation}
We have the basic estimate
\[\P(|B_1-B_{-1}|<\varepsilon)\underset{\varepsilon\to0}{\sim}2\varepsilon\int_\R\frac{e^{-x^2}}{2\pi}dx=\frac{\varepsilon}{\sqrt{\pi}}.\]

Now we need to estimate the same probability when conditioned on $\left\lbrace\underset{-1\leq t\leq 1}{\sup}B_t<x\right\rbrace$. On this event, the process $(x-B_t)_{-1\leq t\leq 1}$ has the law of a two-sided Bessel process started from $x$. At time 1, the density of this Bessel process is the density of $((x+X)^2+Y^2+Z^2))^{1/2}$ where $(X,Y,Z)$ are i.i.d. normal random variables. Let $f_x(\cdot)$ be the density function of this random variable. It is straightforward to check that $f_0(r)=\sqrt{\frac{2}{\pi}}r^2e^{-\frac{r^2}{2}}\ind_{u\geq0}$ and furthermore
\[\int_0^\infty f_0(r)^2dr=\frac{2}{\pi}\int_0^\infty r^4e^{-r^2}dr=\frac{3}{4\sqrt{\pi}}.\]
 Now we have the following bounds on $f_x$ (recall $x\geq0$)
\[\sqrt{\frac{2}{\pi}}r^2e^{-\frac{(r+x)^2}{2}}\leq f_x(r)\leq \sqrt{\frac{2}{\pi}}r^2e^{-\frac{(r-x)^2}{2}},\]
so that 
\[\int_0^\infty f_x(r)^2dr=\frac{3}{4\sqrt{\pi}}+o_x(1).\]
From here a straightforward computation shows
\[\underset{\varepsilon\to0}{\lim}\frac{\P\left(|B_1-B_{-1}|<\varepsilon\left|\underset{-1\leq t\leq 1}{\sup}B_t<x\right.\right)}{\P(|B_1-B_{-1}|<\varepsilon)}=\frac{\int_0^\infty f_0(r)^2dr}{\int_\R\frac{e^{-r^2}}{2\pi}dr}+o_x(1)=\frac{3}{2}+o_x(1).\]

Hence recalling \eqref{eq:estimate}:
\[\P\left(\left.\underset{-1\leq t\leq 1}{\sup}B_t<x\right|B_1=B_{-1}\right)=\underset{\varepsilon\to0}{\lim}\P\left(\left.\underset{-1\leq t\leq 1}{\sup}B_t<x\right||B_1-B_{-1}|<\varepsilon\right)\underset{x\to0}{\sim}\frac{3}{\pi}x^2.\]

%

\end{proof}
Let us see how the Lemma is useful. Let $(B_t)_{-1\leq t\leq 1}$ be standard two-sided Brownian motion. Then the even part $B^e_t:=\frac{B_t+B_{-t}}{\sqrt{2}}$ and the odd part $B^o_t:=\frac{B_t-B_{-t}}{\sqrt{2}}$ are independent Brownian motions, and $|B_1-B_{-1}|=\sqrt{2}|B^o_1|$. So conditioning on the event $B_1=B_{-1}$ is conditioning on $B^o_1=B^o_{-1}$, i.e. taking the odd part to be a Brownian bridge. Hence if $\tilde{B}_{-1\leq t\leq 1}$ is the radial part of the GFF on $\T_1$, we have $\P\left(\underset{-1\leq t\leq 1}{\sup}\tilde{B}_t<x\right)=\P\left(\left.\underset{-1\leq t\leq 1}{\sup}B_t<x\right|B_1=B_{-1}\right)$. The general case follows by Brownian scaling.

\appendix
\section{The DOZZ formula}
\label{app:dozz}
The DOZZ formula is the expression of the 3-point correlation function on the sphere $\langle V_{\alpha_1}(0)V_{\alpha_2}(1)V_{\alpha_3}(\infty)\rangle_{\S^2}$. The formula reads
\begin{equation}
\label{eq:dozz}
\begin{aligned}
C_\gamma(\alpha_1,\alpha_2,\alpha_3)=&\left(\pi\mu\left(\frac{\gamma}{2}\right)^{2-\frac{\gamma^2}{2}}\frac{\Gamma(\gamma^2/4)}{\Gamma(1-\gamma^2/4)}\right)^{-\frac{\bbar{\alpha}-2Q}{\gamma}}\\
&\times\frac{\Upsilon'_{\frac{\gamma}{2}}(0)\Upsilon_{\frac{\gamma}{2}}(\alpha_1)\Upsilon_{\frac{\gamma}{2}}(\alpha_2)\Upsilon_{\frac{\gamma}{2}}(\alpha_3)}{\Upsilon_{\frac{\gamma}{2}}\left(\frac{\bbar{\alpha}-2Q}{2}\right)\Upsilon_{\frac{\gamma}{2}}\left(\frac{\bbar{\alpha}}{2}-\alpha_1\right)\Upsilon_{\frac{\gamma}{2}}\left(\frac{\bbar{\alpha}}{2}-\alpha_2\right)\Upsilon_{\frac{\gamma}{2}}\left(\frac{\bbar{\alpha}}{2}-\alpha_3\right)},
\end{aligned}
\end{equation}
where $\bbar{\alpha}=\alpha_1+\alpha_2+\alpha_3$ and $\Upsilon_{\frac{\gamma}{2}}$ is Zamolodchikov's special function. It has the following integral representation for $\Re z\in(0,Q)$ 
\[\log\U(z)=\int_0^\infty\left(\left(\frac{Q}{2}-z\right)^2e^{-t}-\frac{\sinh^2\left(\left(\frac{Q}{2}-z\right)\frac{t}{2}\right)}{\sinh\left(\frac{\gamma t}{4}\right)\sinh\left(\frac{t}{\gamma}\right)}\right)\frac{dt}{t}\]
and it extends holomorphically to $\C$.

It satisfies the functional relation $\U(Q-z)=\U(z)$ and it has a simple zero at $0$ if $\gamma^2\in\R\setminus\Q$\footnote{This is not really a restriction since the theory is continuous in $\gamma$}. Thus it also has a simple zero at $Q$ and $\U'(Q)=-\U'(0)\neq0$.

Of great importance in this paper is the derivative DOZZ formula at the critical point $\alpha_1=Q=\alpha_3$ which has the expression
\[\partial_{\alpha_1\alpha_3}^2C_\gamma(Q,\alpha,Q)=\left(\pi\mu\left(\frac{\gamma}{2}\right)^{2-\frac{\gamma^2}{2}}\frac{\Gamma(\gamma^2/4)}{\Gamma(1-\gamma^2/4)}\right)^{-\frac{\alpha}{\gamma}}\frac{\Upsilon_\frac{\gamma}{2}'(0)^3\Upsilon_\frac{\gamma}{2}(\alpha)}{\Upsilon_\frac{\gamma}{2}\left(\frac{\alpha}{2}\right)^4}.\]

\section{Conical singularities}
\label{app:conical}

\begin{figure}[h!]
\centering
\includegraphics[scale=0.6]{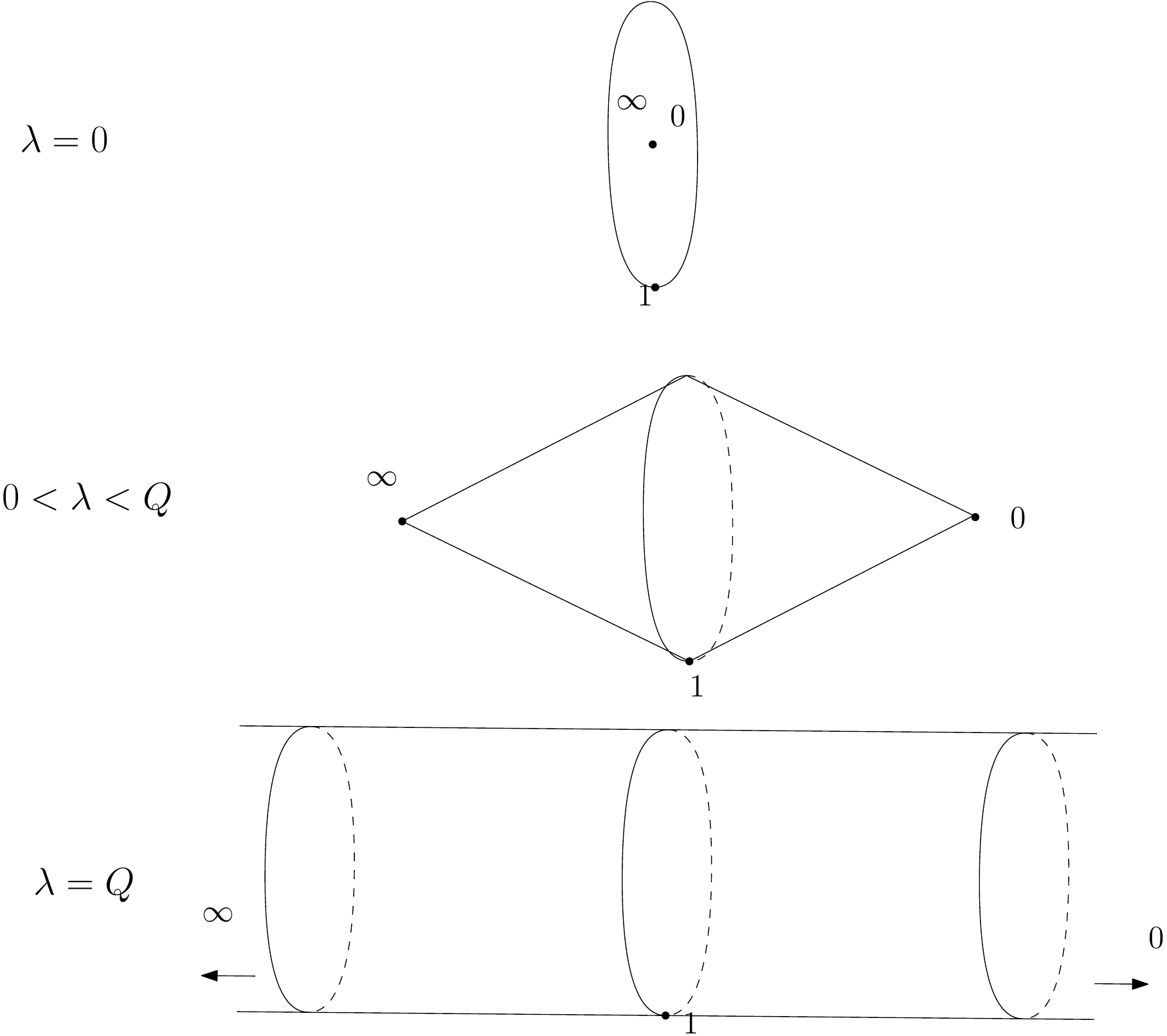}
\caption{\label{fig:conical}Conic degeneration under the insertion of the vertex operators $V_\lambda(0)$ and $V_\lambda(\infty)$. \textbf{Top}: For $\lambda=0$, we have the cr\^epe metric, i.e. two disks glued together. \textbf{Middle}: For $0<\lambda<Q$, we have two Euclidean cones glued together. \textbf{Bottom}: For $\lambda=Q$, the angle of the cones is 0, so we get a bi-infinite cylinder. The limit $\lambda\to Q$ is the setting of the proof of Proposition \ref{prop:spheric_limit}}
\end{figure}

We study the effect of a change of measure with respect to the Liouville field. Let $X$ be a GFF on $\S^2$\footnote{We work on the sphere for concreteness but this argument is valid on any compact Riemann surface.} with some background metric $g$ and $dM^\gamma_g$ be the associated chaos measure (regularised in $g$). Let $\omega\in H^1_0$ be a function such that $e^{\frac{Q}{2}\omega}\in L^1(dM^\gamma)$. Let $\hat{g}:=e^{\omega}g$ and $dM^\gamma_{\hat{g}}$ be the chaos of $X$ regularised in $\hat{g}$. Then for all $r>0$, applying succesively Girsanov's theorem and conformal covariance, we find
\begin{equation}
\E\left[e^{\langle X,\frac{Q}{2}\omega\rangle_\nabla-\frac{Q^2}{8}\norm{\omega}_\nabla^2}M^\gamma(\S^2)^{-r}\right]=\E\left[\left(\int_{\S^2}e^{\frac{\gamma Q}{2}\omega}dM^\gamma_g\right)^{-r}\right]=\E\left[M^\gamma_{\hat{g}}(\S^2)^{-r}\right].
\end{equation}
In particular, the vertex operator which is formally written $V_\alpha(z)=e^{\alpha X(z)-\frac{\alpha^2}{2}\E[X(z)^2]}$ is a special case of the previous setting with $\omega=\frac{2\alpha}{Q}G(z,\cdot)$. Hence, after regularising, we find that adding a vertex operator is the same as conformally multiplying the metric and set $\hat{g}=e^{\frac{2\alpha}{Q}G(z,\cdot)}g$, i.e. the new metric satisfies $\log\hat{g}(z+h)=-\frac{2\alpha}{Q}\log |h|+O_h(1)$ so it has a conical singularity of order $\alpha/Q$ at $z$. 

Another way to see this is to look at the curvature, which reads in the distributional sense
\[K_{\hat{g}}=e^{-\frac{2\alpha}{Q}G(z,\cdot)}\left(K_g+\frac{4\pi\alpha}{Q}\left(\delta_z-\frac{1}{\Vol_g(\S^2)}\right)\right),\]
where $\Vol_g(\S^2)$ is the volume of the sphere in the metric $g$. Thus the metric has an atom of curvature at $z$, meaning it has a conical singularity.

If $\alpha=Q$, the singularity is no longer integrable, so the volume is infinite and the surface has a semi-infinite cylinder. Loosely, we will refer to this situation as a cusp, even though the hyperbolic cusp has finite volume because of the extra $\log$-correction in the metric:
\[\log\hat{g}(z+h)=-2\log|h|-2\log\log\frac{1}{|h|}+O(1).\]
The reason for this abuse of terminology is that we are interested in GMC measure. Indeed, suppose $z=0$ in the sphere coordinates. By conformal covariance, if we use the cylinder coordinates, the $\log$-correction term is the same as shifting the radial part of the GFF from the Brownian motion $(B_s)_{s\geq0}$ to $(B_s-Q\log(1+s))_{s\geq0}$. Up to time $t$, this corresponds to a change of measure given by the exponential martingale $e^{-Q\int_0^t\frac{dB_s}{1+s}-\frac{Q^2}{2}\int_0^t\frac{1}{(1+s)^2}ds}$, which is uniformly integrable since $\int_0^\infty\frac{1}{(1+t)^2}dt<\infty$. So the new field is absolutely continuous with respect to the old one, meaning that GMC does not make a difference between a Euclidean cylinder and a hyperbolic cusp.


\hspace{10 cm}

 \end{document}